\documentclass[12pt,reqno,a4paper]{article}

\usepackage{amsmath,amsthm}
\usepackage{amssymb}
\usepackage{mathrsfs,enumerate}
\usepackage{slashbox}
\usepackage{mathtools} 
\usepackage{color} 

\newtheorem{theorem}{Theorem}[section]
\newtheorem{prop}[theorem]{Proposition}
\newtheorem{cor}[theorem]{Corollary}
\newtheorem{conj}[theorem]{Conjecture}
\newtheorem{lemma}[theorem]{Lemma}
\newtheorem{definition}[theorem]{Definition}

\theoremstyle{remark}
\newtheorem{remark}[theorem]{Remark}
\newtheorem{example}[theorem]{Example}

\def\Li{{\rm Li}}
\def\aa{{\mathscr{A}}}
\def\ee{{\mathcal{E}}}
\def\EE{{\mathbb{E}}}

\def\Q{\mathbb{Q}}
\def\TT{{\widetilde{T}}}
\def\TTcal{\widetilde{\mathcal{T}}}
\DeclareFontEncoding{OT2}{}{}
\DeclareFontSubstitution{OT2}{cmr}{m}{n}
\DeclareSymbolFont{cyss}{OT2}{wncyss}{m}{n}
\DeclareMathSymbol{\sh}{\mathbin}{cyss}{`x}
\def\Z{\mathbb{Z}}
\allowdisplaybreaks

\begin{document}

\title{Multiple $L$-values of level four, poly-Euler numbers, and related zeta functions} 

\author{Masanobu Kaneko and Hirofumi Tsumura}



\date{}
\maketitle

\begin{abstract}
We present several formulas for some specific multiple $L$-values of conductor four.
This grew out from the study of zeta functions of level four of Arakawa-Kaneko type.
Closely related is a new version of multiple poly-Euler numbers and we briefly discuss this too.
\end{abstract}

\section{Introduction}\label{sec:intro}

This paper is in a sense a continuation of our previous work \cite{KT2018, KT2020-ASPM, KT2020-Tsukuba}.
There we studied a certain type of zeta functions and their values at positive as well as negative integer arguments,
the prototype being the so-called Arakawa-Kaneko zeta function investigated in \cite{AK1999}, whose values at positive integers
can be written in terms of multiple zeta values and at negative integers in terms of poly-Bernoulli numbers.  

Through our study in this context in the case of level 4, we are naturally
led to the investigation of a class of multiple $L$-values of level (or conductor) 4, which in our notation is given by
\[  L_\sh(k_1,\ldots,k_r;\chi_4,\ldots,\chi_4)=\sum_{1\leq m_1<\cdots<m_r \atop m_j \equiv j \bmod 2} \frac{(-1)^{(m_r-r)/2}}{m_1^{k_1}\cdots m_r^{k_r}}. \]
We discovered that these $L$-values with special indices satisfy a notable linear relation with combinatorial numbers (`Entringer 
numbers') as coefficients. Also we found a (conjectural) connection of double $L$-values to modular forms of level four, a relation
of previously studied level 2 $L$-values (`multiple $T$-values') to our level 4 values, 
and a formula for the generating function of `height one' $L$-values in terms of the Appell hypergeometric
function $F_1$, being in contrast to Gauss's ${}_2F_1$ in the previous cases. 

In the next section, we present these findings on the level 4 multiple $L$-values in detail. Various iterated integral expressions are our basic tools.
In Section~\ref{sec:zeta-polyEuler}, we introduce and study the motivating Arakawa-Kaneko type zeta function of level 4
and the corresponding multiple poly-Euler numbers (different from those introduced and studied in Sasaki \cite{Sasaki2012} 
and Ohno-Sasaki \cite{OS2012,OS2013,OS2017}).  In the final Section~\ref{sec:zetafunct} we connect the one variable multiple
$\TT$-function to the zeta function introduced in \S\ref{sec:zeta-polyEuler} and obtain several family of relations among multiple
$\TT$-values. The formulas and the techniques used in these two sections are more or less parallel to our previous studies.

\section{Some results and conjectures on multiple $L$-values of level four}\label{sec:ttilde}

\subsection{Definition}\label{subsec:def}
 
In \cite{AK2004}, Arakawa and the first named author defined two types of multiple $L$-values as follows. Let $f_j\,:\,\mathbb{Z} \to \mathbb{C}$ $(j=1,\ldots,r)$ be some periodic functions
(in the sequel we only consider the Dirichlet character of conductor 4 and its square). 
For integers $k_1,\ldots,k_r\in \mathbb{Z}_{\geq 1}$, define
\begin{align}
& L_\sh(k_1,\ldots,k_r;f_1,\ldots,f_r)=\sum_{1\leq m_1<\cdots<m_r} \frac{f_1(m_1)f_2(m_2-m_1)\cdots f_r(m_r-m_{r-1})}{m_1^{k_1}m_2^{k_2}\cdots m_r^{k_r}},  \label{MLV-sh}\\
& L_\ast(k_1,\ldots,k_r;f_1,\ldots,f_r)=\sum_{1\leq m_1<\cdots<m_r} \frac{f_1(m_1)f_2(m_2)\cdots f_r(m_r)}{m_1^{k_1}m_2^{k_2}\cdots m_r^{k_r}}. \label{MLV-star}
\end{align}
If $k_r\geq 2$, these series are absolutely convergent, and if $k_r=1$, these are interpreted as 
\begin{align*}
& L_\sh(k_1,\ldots,k_r;f_1,\ldots,f_r)=\lim_{M\to \infty}\sum_{1\leq m_1<\cdots<m_r<M} \frac{f_1(m_1)f_2(m_2-m_1)\cdots f_r(m_r-m_{r-1})}{m_1^{k_1}m_2^{k_1}\cdots m_r^{k_r}},  \\
& L_\ast(k_1,\ldots,k_r;f_1,\ldots,f_r)=\lim_{M\to \infty}\sum_{1\leq m_1<\cdots<m_r<M} \frac{f_1(m_1)f_2(m_2)\cdots f_r(m_r)}{m_1^{k_1}m_2^{k_1}\cdots m_r^{k_r}},  
\end{align*}
when the limits exist. As the notation suggests, these values satisfy (under a suitable condition, see~\cite{AK2004}) shuffle and stuffle (or harmonic) product rules respectively.

Let $\chi_4$ be the (unique) primitive Dirichlet character of conductor $4$, and consider the case of $f_j=\chi_4$ for all $j$.
Then it is readily seen by definition that 
\begin{equation}
 L_\sh(k_1,\ldots,k_r;\chi_4,\ldots,\chi_4)=\sum_{1\leq m_1<\cdots<m_r \atop m_j \equiv j \bmod 2} \frac{(-1)^{(m_r-r)/2}}{m_1^{k_1}\cdots m_r^{k_r}},\label{MLV-chi-4} 
\end{equation}
which is convergent even when $k_r=1$.
To ease notation, and introducing the factor $2^r$ for later convenience, we define the `multiple $\TT$-values' for any tuple of 
positive integers $(k_1,\ldots,k_r)$  
\begin{equation}\label{ttilde-def} \TT(k_1,\ldots,k_r):=2^r L_\sh(k_1,\ldots,k_r;\chi_4,\ldots,\chi_4)
=2^r\sum_{1\leq m_1<\cdots<m_r \atop m_j \equiv j \bmod 2} \frac{(-1)^{(m_r-r)/2}}{m_1^{k_1}\cdots m_r^{k_r}}. \end{equation}
This is in contrast to our previously studied object 
\begin{equation}\label{t-def} T(k_1,\ldots,k_r)=2^rL_\sh(k_1,\ldots,k_r;\chi_4^2,\ldots,\chi_4^2) 
=2^r\sum_{1\leq m_1<\cdots<m_r \atop m_j \equiv j \bmod 2} \frac{1}{m_1^{k_1}\cdots m_r^{k_r}},\end{equation}
which we called `multiple $T$-values.'  For this, the last  entry $k_r$ should be larger than 1.  More generally, we introduced in~\cite{KT2020-ASPM, KT2020-Tsukuba} 
a level 2 analogue of the multiple polylogarithms for $k_1,\ldots,k_r\in \mathbb{Z}_{\ge1}$:
\begin{equation}
A(k_1,\ldots,k_r;z)=2^r\sum_{0<m_1<\cdots <m_r\atop m_i\equiv i\bmod 2}
\frac{z^{m_r}}{m_1^{k_1}\cdots m_{r}^{k_{r}}}\quad (|z|<1).  \label{Def-A-polylog}
\end{equation}
In particular, 
\begin{equation}\label{A1exp} 
A(1;z)=2\sum_{n=0}^\infty \frac{z^{2n+1}}{2n+1}=2\tanh^{-1}(z)=\log\left(\frac{1+z}{1-z}\right)=\int_0^z 
\frac{2dt}{1-t^2}.
\end{equation}
As is easily seen (parallel to the case of usual multiple polylogarithms), the function $A(k_1,\ldots,k_r;z)$
satisfies the derivative formula (see \cite[Lemma 5.1]{KT2020-ASPM})
\begin{align}\label{deriv-a-multi}
\frac{d}{dz}A(k_1,\ldots,k_r;z)
& =
\begin{cases}
\displaystyle{\frac{1}{z}}\,A(k_1,\ldots,k_{r-1},k_r-1;z) & (k_r\geq 2),\\
\displaystyle{\frac{2}{1-z^2}}\,A(k_1,\ldots,k_{r-1};z) & (k_r=1),
\end{cases}
\end{align}
and hence $A(k_1,\ldots,k_r;z)$ is realized as an iterated integral starting with \[ A(1;z)=\int_0^z \frac{2dt}{1-t^2}\]  as follows.
We introduce a compact notation of writing an iterated integral. 
Let $a$ and $b$ be points in the closed unit disk. For differential 1-forms
$ \Omega_1, \Omega_2,\ldots,\Omega_k$, we understand by the expression 
\[ \int_a^b \Omega_1\circ  \Omega_2\circ \cdots\cdots \circ\Omega_k\]
the iterated integral
\[ \int_a^b\left(\int_a^{t_k}\cdots\left(\int_a^{t_3}\left(\int_a^{t_2}\Omega_1\right)\Omega_2\right)\cdots\cdots\Omega_k\right),\]
where $\Omega_i=\Omega_i(t_{i})dt_i$. 
Here and in the following, every path of integration is considered inside the unit disk. Under this notation, from \eqref{A1exp}
and \eqref{deriv-a-multi}, we immediately have
\begin{equation*} 
A(k_1,\ldots,k_r;z)= \int_0^z \Omega_1\circ \Omega_2\circ \cdots\cdots \circ\Omega_k,
\end{equation*}
where $\Omega_j=2dt_j/(1-t_j^2)$ if $j\in\{1,k_1+1,\ldots,k_1+\cdots+k_{r-1}+1\}$ and $\Omega_j=dt_j/t_j$
otherwise.  Note that the total number $k$ of differential forms is the {\it weight} $k_1+\cdots+k_r$ of the index
$(k_1,\ldots,k_r)$, and the number of $2dt/(1-t^2)$ is the {\it depth} $r$, the total number of components of the index. 

Since
\begin{equation}
T(k_1,\ldots,k_r)=A(k_1,\ldots,k_r;1)\ \ (k_r\geq 2),\label{A-T-rel}
\end{equation}
we have  an integral expression of multiple $T$-values (see also \cite[Theorem~2.1]{KT2020-Tsukuba})
\begin{align}
&T(k_1,k_{2},\ldots,k_{r})  \nonumber  \\
&=\int_0^1\frac{2dt}{1-t^2}\circ\underbrace{\frac{dt}{t}\circ\cdots\circ\frac{dt}{t}}_{k_1-1}\circ
\frac{2dt}{1-t^2}\circ\underbrace{\frac{dt}{t}\circ\cdots\circ\frac{dt}{t}}_{k_2-1}\circ\cdots\cdots\circ
\frac{2dt}{1-t^2}\circ\underbrace{\frac{dt}{t}\circ
\cdots\circ\frac{dt}{t}}_{k_r-1} \nonumber \\
&=\mathop{\int\cdots\int}\limits_{0<t_1<\cdots <t_k<1}\frac{2dt_1}{1-t_1^2}\underbrace{\frac{dt_2}{t_2}\cdots\frac{dt}{t}}_{k_1-1}
\frac{2dt}{1-t^2}\underbrace{\frac{dt}{t}\cdots\frac{dt}{t}}_{k_2-1}\cdots\cdots\frac{2dt}{1-t^2}\underbrace{\frac{dt}{t}
\cdots\frac{dt_{k}}{t_{k}}}_{k_r-1}.\label{T-integ-exp}
\end{align}
For simplicity, we have suppressed the subscripts of variables of most of the differential forms.

In the same vein, the multiple $\TT$-values can also be given as integrals. This fact is fundamental to almost all of our proofs of Theorems.
First, as in the case of usual multiple zeta values and $T$-values, by expanding $1/(1+t^2)$ into geometric series and 
integrating term by term, we obtain:  
\begin{prop}
For $k_1,\ldots,k_r\in \mathbb{Z}_{\geq 1}$, we have 
\begin{align}
&\TT(k_1,k_{2},\ldots,k_{r})  \nonumber  \\ 
&=\int_0^1\frac{2dt}{1+t^2}\circ\underbrace{\frac{dt}{t}\circ\cdots\circ\frac{dt}{t}}_{k_1-1}\circ
\frac{2dt}{1+t^2}\circ\underbrace{\frac{dt}{t}\circ\cdots\circ\frac{dt}{t}}_{k_2-1}\circ\cdots\cdots\circ
\frac{2dt}{1+t^2}\circ\underbrace{\frac{dt}{t}\circ
\cdots\circ\frac{dt}{t}}_{k_r-1} \nonumber \\
&=\mathop{\int\cdots\int}\limits_{0<t_1<\cdots <t_k<1}\frac{2dt_1}{1+t_1^2}\underbrace{\frac{dt_2}{t_2}\cdots\frac{dt}{t}}_{k_1-1}
\frac{2dt}{1+t^2}\underbrace{\frac{dt}{t}\cdots\frac{dt}{t}}_{k_2-1}\cdots\cdots\frac{2dt}{1+t^2}\underbrace{\frac{dt}{t}
\cdots\frac{dt_{k}}{t_{k}}}_{k_r-1}.\label{integ-exp}
\end{align}
\end{prop}

A typical consequence, which we use later several times, is the following identity.

\begin{cor} For any $n\ge1$, we have 
\begin{equation}\label{11111}
\frac{\TT(1)^n}{n!}=\TT(\underbrace{1,\ldots,1}_n). 
\end{equation}
\end{cor}

If we make a change of variables $t\to (1-u)/(1+u)$ appeared in \cite{KT2020-Tsukuba} for $T$-values,
we obtain from \eqref{integ-exp} another integral expression when $k_r>1$ (in the case of `admissible' index). For the definition of the `dual' of an index,
see for instance \cite[\S3.1]{KT2020-Tsukuba}.

\begin{cor}
Suppose $k_r>1$ and let $(l_1,\ldots,l_s)$ be the dual index of $(k_1,k_{2},\ldots,k_{r})$.  Then we have
\begin{align}
&\TT(k_1,k_{2},\ldots,k_{r}) \nonumber\\
&=\int_0^1\frac{2du}{1-u^2}\circ\underbrace{\frac{2du}{1+u^2}\circ\cdots\circ\frac{2du}{1+u^2}}_{l_1-1}\circ
\frac{2du}{1-u^2}\circ\underbrace{\frac{2du}{1+u^2}\circ\cdots\circ\frac{2du}{1+u^2}}_{l_2-1}\circ\cdots \nonumber\\
& \qquad\qquad\qquad\cdots\circ
\frac{2du}{1-u^2}\circ\underbrace{\frac{2du}{1+u^2}\circ
\cdots\circ\frac{2du}{1+u^2}}_{l_s-1} \nonumber \\
&=\mathop{\int\cdots\int}\limits_{0<u_1<\cdots <u_k<1}\frac{2du_1}{1-u_1^2}\underbrace{\frac{2du_2}{1+u_2^2}\cdots\frac{2du}{1+u^2}}_{l_1-1}
\frac{2du}{1-u^2}\underbrace{\frac{2du}{1+u^2}\cdots\frac{2du}{1+u^2}}_{l_2-1}\cdots   \label{integ-exp2} \\
&\qquad\qquad\qquad\cdots\frac{2du}{1-u^2}\underbrace{\frac{2du}{1+u^2}\cdots\frac{2du_k}{1+u_k^2}}_{l_s-1}. \nonumber
\end{align}
\end{cor}
Furthermore, by the series expressions~\eqref{ttilde-def} and~\eqref{Def-A-polylog}, we see that
\begin{equation}\label{TtbyA}
 \TT(k_1,\ldots,k_r)=i^{-r}A(k_1,\ldots,k_r;i)\quad (i=\sqrt{-1}), 
 \end{equation}
and hence we have yet another iterated integral representation of $\TT(k_1,\ldots,k_r)$:
\begin{align}\label{integ-exp3}
&\TT(k_1,\ldots,k_r) \nonumber\\
&=i^{-r} \int_0^i\frac{2dt}{1-t^2}\circ\underbrace{\frac{dt}{t}\circ\cdots\circ\frac{dt}{t}}_{k_1-1}\circ
\frac{2dt}{1-t^2}\circ\underbrace{\frac{dt}{t}\circ\cdots\circ\frac{dt}{t}}_{k_2-1}\circ\cdots\cdots\circ
\frac{2dt}{1-t^2}\circ\underbrace{\frac{dt}{t}\circ
\cdots\circ\frac{dt}{t}}_{k_r-1}. 
\end{align}

\subsection{The space of multiple $\TT$-values}\label{subsec:space}

As usual, let us consider the $\mathbb{Q}$-vector space 
\[ \TTcal=\sum_{k=0}^\infty\,\TTcal_k\] spanned by all multiple $\TT$-values, where
\[  \TTcal_0=\mathbb{Q},\quad 
\TTcal_{k}=\sum_{1\leq r \leq k \atop {{k_1,\ldots,k_{r}\geq 1 \atop k_1+\cdots+k_r=k}}}
\mathbb{Q}\cdot \TT(k_1,\ldots,k_r)\quad (k\geq 1). 
\]

\begin{prop}  The space $\TTcal$ is a $\Q$-algebra under the usual multiplication of real numbers.
\end{prop}
\begin{proof}
This is a standard consequence of the integral expression~\eqref{integ-exp}, the product being
described by the {\it shuffle product} rule.
\end{proof}

\begin{example}
By~\eqref{integ-exp}, the shuffle product of $\TT$-values takes exactly the same form
as in the case of multiple zeta (and $T$-) values. In our case, the $\TT$-values for non-admissible indices 
also converge and we do not need any regularization procedure. For instance, we have
\begin{align*}
\TT(1)\TT(2)&=2\TT(1,2)+\TT(2,1),\\
\TT(2)^2&=4\TT(1,3)+2\TT(2,2).
\end{align*}
\end{example}
 
The first natural question would be the dimension $d_k$ 
over $\Q$ of each subspace $\TTcal_k$ of weight $k$ elements.  
We have conducted numerical experiments with Pari-GP,
and obtained the following conjectural table.

\begin{center}
\begin{tabular}{|c|c|c|c|c|c|c|c|c|c|c|c|c|c|c|c|c|c|} \hline
$k$ & 0&1&2& 3& 4& 5& 6& 7& 8& 9& 10& 11\\ 
\hline  
$d_k$ & 1& 1& 2& 3& 6& 8& 16& 22& 44& 59& 118& 162\\ 
\hline
$d_{k,ev}$ & 1& 0& 1& 2& 3& 4& 8& 12& 22& 30& 59& 84\\ 
\hline
$d_{k,od}$ & 0& 1& 1& 1& 3& 4& 8& 10& 22& 29& 59& 78\\ 
\hline
\end{tabular}
\end{center}
\vspace{7pt}
It seems that the space of multiple $\TT$-values of {\it even} depth and that of {\it odd} depth
are disjoint. In the table above we also display the conjectural dimensions $d_{k,ev}$ and $d_{k,od}$ of the spaces of even and odd
depth multiple $\TT$-values of weight $k$ respectively.  As far as the table made by numerical experiments up to weight $11$ goes,
the predicted relation $d_k=d_{k,ev}+d_{k,od}$ holds.
We can also read off the relation $d_k=2d_{k-1}$ and $d_{k,ev}=d_{k,od}$ if $k>1$ is even in this range, 
but we are not aware of any reason to believe that this holds true in general.

Through the numerical experiments, we found it very likely that the multiple $T$-values lie in the space $\TTcal$ 
of multiple $\TT$-values and moreover in the subspace spanned by {\it even} depth multiple $\TT$-values.
This prediction, first announced as a conjecture in a conference, was confirmed soon after (independently) by  
R.~Umezawa and M.~Hirose. They pointed out that, essentially, this follows from the path-composition formula for the iterated integrals.
We give here seemingly the most direct proof using that formula.

\begin{theorem}[Umezawa~\cite{U}]\label{Teven} 
Any multiple $T$-value can be written as a linear combination of multiple $\TT$-values 
of even depth.
\end{theorem}

\begin{proof}
We use the formula \eqref{T-integ-exp}: 
\begin{equation*} 
T(k_1,\ldots,k_r)= \int_0^1 \Omega_1\circ \Omega_2\circ \cdots\cdots \circ\Omega_k,
\end{equation*}
where $\Omega_j=2dt_j/(1-t_j^2)$ if $j\in\{1,k_1+1,\ldots,k_1+\cdots+k_{r-1}+1\}$ and $\Omega_j=dt_j/t_j$
otherwise.  Now, by decomposing the path of integration into two paths, first from 0 to $i=\sqrt{-1}$ and from $i$ to 1,
and using the path composition formula for iterated integrals (see for instance \cite[Prop.~1.5.1]{Chen}), we may write this as
\[ T(k_1,\ldots,k_r)=\sum_{j=0}^k \left(\int_0^i \Omega_1\circ \cdots \circ\Omega_j\right)\left(\int_i^1 \Omega_{j+1}\circ
\cdots\circ\Omega_k\right). \]
Suppose that $\int_0^i \Omega_1\circ \cdots \circ\Omega_j$ has depth $l$, i.e., there are $l$ $2dt/(1-t^2)$'s 
among $ \Omega_1, \ldots, \Omega_j$. Then by \eqref{integ-exp3}, $\int_0^i \Omega_1\circ \cdots \circ\Omega_j$
is equal to  $i^l$ times a multiple $\TT$-value of depth $l$.   
  
On the other hand, by changing the variables $t\to (-iu+1)/(u-i)$, we see that the integral  
$\int_i^1 \Omega_{j+1}\circ\cdots\circ\Omega_k$ is equal to $i^{-(k-j-r+l)}$ times a real iterated integral
$\int_0^1 \Omega_{j+1}'\circ\cdots\circ\Omega_k'$, where $\Omega_h'=2du/(1-u^2)$ if 
$\Omega_h=2dt/(1-t^2)$ and $\Omega_h'=2du/(1+u^2)$ if $\Omega_h=dt/t$.
Since the number of $2du/(1-u^2)$ (depth) among $\Omega_h'$  is $r-l$ and the total number (weight) is $k-j$,
by \eqref{integ-exp2}, we conclude that the integral $\int_i^1 \Omega_{j+1}\circ\cdots\circ\Omega_k$ is equal to $i^{-(k-j-r+l)}$ times 
a multiple $\TT$-value of weight $k-j$ and depth $k-j-r+l$. Therefore, we conclude that
the product $\left(\int_0^i \Omega_1\circ \cdots \circ\Omega_j\right)\left(\int_i^1 \Omega_{j+1}\circ
\cdots\circ\Omega_k\right)$ is $i^{-k+j+r}$ times a product of multiple $\TT$-values of
depths $l$ and $k-j-r+l$, which is by the shuffle product  $i^{-k+j+r}$ times a sum of multiple $\TT$-values of
depth $k-j-r+2l$.  Since $T(k_1,\ldots,k_r)$ is real and the power $-k+j+r$ of $i$ and the depth $k-j-r+2l$
have the same parity, we conclude that the $T$-value $T(k_1,\ldots,k_r)$ is a sum of $\TT$-values of
even depth.
\end{proof}

\begin{example} Consider the `height one' multiple $T$-values, i.e., $T$-values of the form $T(\underbrace{1,\ldots,1}_{r-1},k+1)$ with $r,k\ge1$. 
Starting with the iterated integral expression \eqref{T-integ-exp} and proceeding as in the above proof, we have
\begin{align*}
&T(\underbrace{1,\ldots,1}_{r-1},k+1) 
=\int_0^1\underbrace{\frac{2dt}{1-t^2}\circ\cdots\circ\frac{2dt}{1-t^2}}_{r}\circ
\underbrace{\frac{dt}{t}\circ\cdots\circ\frac{dt}{t}}_{k}\\
&=\left(\int_0^i\underbrace{\frac{2dt}{1-t^2}\circ\cdots\circ\frac{2dt}{1-t^2}}_{r}\right)
\left(\int_i^1 \underbrace{\frac{dt}{t}\circ\cdots\circ\frac{dt}{t}}_{k}\right)\\
&+\sum_{j=1}^r \left(\int_0^i\underbrace{\frac{2dt}{1-t^2}\circ\cdots\circ\frac{2dt}{1-t^2}}_{r-j}\right)
\left(\int_i^1\underbrace{\frac{2dt}{1-t^2}\circ\cdots\circ\frac{2dt}{1-t^2}}_{j}\circ
\underbrace{\frac{dt}{t}\circ\cdots\circ\frac{dt}{t}}_{k}\right)\\
&+\sum_{j=1}^k \left(\int_0^i\underbrace{\frac{2dt}{1-t^2}\circ\cdots\circ\frac{2dt}{1-t^2}}_{r}\circ
\underbrace{\frac{dt}{t}\circ\cdots\circ\frac{dt}{t}}_{j}\right)
\left(\int_i^1 \underbrace{\frac{dt}{t}\circ\cdots\circ\frac{dt}{t}}_{k-j}\right).
\end{align*}
Now, by~\eqref{integ-exp3}, we have 
\begin{align*}
\int_0^i\underbrace{\frac{2dt}{1-t^2}\circ\cdots\circ\frac{2dt}{1-t^2}}_{r-j}
&=i^{r-j} \TT(\underbrace{1,\ldots,1}_{r-j}),\\
\int_0^i\underbrace{\frac{2dt}{1-t^2}\circ\cdots\circ\frac{2dt}{1-t^2}}_{r}\circ
\underbrace{\frac{dt}{t}\circ\cdots\circ\frac{dt}{t}}_{j}
&=i^r \TT(\underbrace{1,\ldots,1}_{r-1},j+1),
\end{align*}
and by the change of variable $t\to (-iu+1)/(u-i)$
\begin{align*}
\int_i^1\underbrace{\frac{2dt}{1-t^2}\circ\cdots\circ\frac{2dt}{1-t^2}}_{j}\circ
\underbrace{\frac{dt}{t}\circ\cdots\circ\frac{dt}{t}}_{k}
&=i^{-k}\int_0^1\underbrace{\frac{2du}{1-u^2}\circ\cdots\circ\frac{2du}{1-u^2}}_{j}\circ
\underbrace{\frac{2du}{1+u^2}\circ\cdots\circ\frac{2du}{1+u^2}}_{k}\\
&=i^{-k} \TT(\underbrace{1,\ldots,1}_{k-1},j+1),\\
\int_i^1\underbrace{\frac{dt}{t}\circ\cdots\circ\frac{dt}{t}}_{k-j}
&=i^{-k+j} \int_0^1\underbrace{\frac{2du}{1+u^2}\circ\cdots\circ\frac{2du}{1+u^2}}_{k-j}\\
&=i^{-k+j} \TT(\underbrace{1,\ldots,1}_{k-j}).
\end{align*}
Putting all these together and using the shuffle product formula
\[ \TT(\underbrace{1,\ldots,1}_{r})
\TT(\underbrace{1,\ldots,1}_{k})=\binom{r+k}{r}\TT(\underbrace{1,\ldots,1}_{r+k}), \] we have
\begin{align}
&T(\underbrace{1,\ldots,1}_{r-1},k+1) = i^{r-k} \binom{r+k}{r}\TT(\underbrace{1,\ldots,1}_{r+k}) \nonumber \\
&\qquad+ i^{r-k}  \sum_{j=1}^r i^{-j} \TT(\underbrace{1,\ldots,1}_{r-j}) \TT(\underbrace{1,\ldots,1}_{k-1},j+1)
+ i^{r-k}  \sum_{j=1}^k i^{j} \TT(\underbrace{1,\ldots,1}_{k-j}) \TT(\underbrace{1,\ldots,1}_{r-1},j+1).\label{TbyTtilde}
\end{align}
Taking the real and imaginary parts of this equation, we have the following set of relations.

When $r+k$ is even, 
\begin{align}
&(-1)^{(k-r)/2}\,T(\underbrace{1,\ldots,1}_{r-1},k+1) =\binom{r+k}{r}\TT(\underbrace{1,\ldots,1}_{r+k}) \nonumber\\
&\quad + \sum_{j=1\atop j:even}^r (-1)^{j/2}\, \TT(\underbrace{1,\ldots,1}_{r-j}) \TT(\underbrace{1,\ldots,1}_{k-1},j+1)
+ \sum_{j=1\atop j:even}^k (-1)^{j/2}\, \TT(\underbrace{1,\ldots,1}_{k-j}) \TT(\underbrace{1,\ldots,1}_{r-1},j+1),\\
&\qquad\sum_{j=1\atop j:odd}^r (-1)^{(j+1)/2} \,\TT(\underbrace{1,\ldots,1}_{r-j}) \TT(\underbrace{1,\ldots,1}_{k-1},j+1)\nonumber\\
&\qquad\qquad\qquad + \sum_{j=1\atop j:odd}^k (-1)^{(j-1)/2}\, \TT(\underbrace{1,\ldots,1}_{k-j}) \TT(\underbrace{1,\ldots,1}_{r-1},j+1)=0,
\end{align}
and when $r+k$ is odd,
\begin{align}
(-1)^{(k-r-1)/2}\,T(\underbrace{1,\ldots,1}_{r-1},k+1)  &= \sum_{j=1\atop j:odd}^r (-1)^{(j+1)/2}\, \TT(\underbrace{1,\ldots,1}_{r-j}) \TT(\underbrace{1,\ldots,1}_{k-1},j+1) \nonumber \\
&+ \sum_{j=1\atop j:odd}^k (-1)^{(j-1)/2}\, \TT(\underbrace{1,\ldots,1}_{k-j}) \TT(\underbrace{1,\ldots,1}_{r-1},j+1),
\end{align}
\begin{align}
\binom{r+k}{r}\TT(\underbrace{1,\ldots,1}_{r+k})&+\sum_{j=1\atop j:even}^r (-1)^{j/2}\, \TT(\underbrace{1,\ldots,1}_{r-j}) \TT(\underbrace{1,\ldots,1}_{k-1},j+1)\nonumber \\
&\qquad + \sum_{j=1\atop j:even}^k (-1)^{j/2}\, \TT(\underbrace{1,\ldots,1}_{k-j}) \TT(\underbrace{1,\ldots,1}_{r-1},j+1)=0.
\end{align}

For small values of $r$ and $k$ with $r+k\le 4$, we obtain from these, together with the shuffle product, the followings:
\begin{align*}
T(2)&=2\TT(1,1),\\
T(3)&=T(1,2)=\TT(1,2)+\TT(2,1),\\
\TT(3)&=3\TT(1,1,1),\\
T(4)&=T(1,1,2)=2\TT(1,3)+\TT(2,2)+\TT(3,1)-4\TT(1,1,1,1),\\
\TT(4)&=2\TT(1,1,2)+2\TT(1,2,1)+\TT(2,1,1),\\
T(1,3)&=-2\TT(1,3)+6\TT(1,1,1,1).
\end{align*}

The particular case of $r=1$ (multiplied by $i^k$) of \eqref{TbyTtilde} will be used later in the proof of Theorem~\ref{Ent-rel}:
\begin{align}\label{key}
i^kT(k+1)&=i(k+1)\TT(\underbrace{1,\ldots,1}_{k+1})+\TT(\underbrace{1,\ldots,1}_{k-1},2)
+\sum_{j=1}^k i^{j+1} \TT(\underbrace{1,\ldots,1}_{k-j}) \TT(j+1).
\end{align}
\end{example}

\subsection{Height one multiple $\TT$-values}\label{subsec:ht1}

As another application of the iterated integral expression \eqref{integ-exp}, 
we can compute the generating function of `height one' multiple $\TT$-values. 
Analogous results are known in the case of multiple zeta values as well as multiple $T$-values (\cite{Ao, Dr, KT2020-Tsukuba}):
\[ 1-\sum_{m,n=1}^\infty \zeta(\underbrace{1,\ldots,1}_{n-1},m+1)X^m Y^n
=\frac{\Gamma(1-X)\Gamma(1-Y)}{\Gamma(1-X-Y)}={}_2F_1(X,Y;1;1)^{-1}\]
and
\[  1-\!\!\sum_{m,n=1}^\infty T({1,\ldots,1}_{n-1},m+1)X^m Y^n=
\frac{2\,\Gamma(1-X)\Gamma(1-Y)}{\Gamma(1-X-Y)}{}_2F_1(1-X,1-Y;1-X-Y;-1), \]
where ${}_2F_1(a,b;c;z)$ is the Gauss hypergeometric function.
In contrast to these cases, now the Appell hypergeometric function $F_1$ emerges.

\begin{theorem}\label{ht1gen}
We have 
\[ \sum_{m, n\ge1}\TT(\underbrace{1,\cdots,1}_{n-1},m)X^{m-1}Y^{n-1}
=\frac2{1-X}F_1(1-X;1-iY,1+iY;2-X;i,-i), \]
where $F_1$ stands for the Appell hypergeometric function
\[ F_{1}(a;b_{1},b_{2};c;x,y)=\sum _{m,n=0}^{\infty }{\frac {(a)_{m+n}(b_{1})_{m}(b_{2})_{n}}{(c)_{m+n}\,m!\,n!}}\,x^{m}y^{n}. \]
\end{theorem}

\begin{proof}
From the integral expression~\eqref{integ-exp}, we have
\begin{align*}
\TT(\underbrace{1,\ldots,1}_{n-1},m)&=\underset{0<t_1<\cdots<t_n<u_1<\cdots<u_{m-1} <1}{\int\cdots\int}
\frac{2dt_1}{1+t_1^2}\cdots\frac{2dt_n}{1+t_n^2}\frac{du_1}{u_1}\cdots\frac{du_{m-1}}{u_{m-1}}  \\
&=\int_0^1\left(\frac1{(n-1)!}\left(\int_0^{t_n}\frac{2dt}{1+t^2}\right)^{n-1}\cdot\frac1{(m-1)!}\left(\int_{t_n}^1\frac{du}{u}\right)^{m-1}\right)
\frac{2dt_n}{1+t_n^2}.
\end{align*}
Here, by using 
\[ \int_0^{t_n}\frac{2dt}{1+t^2}=\frac1{i}\log\left(\frac{1+it_n}{1-it_n}\right)\quad(\text{principal value, }0\le t_n\le1), \]
we obtain 
\[ \TT(\underbrace{1,\ldots,1}_{n-1},m)=\frac{i^{1-n}}{(n-1)!(m-1)!}\int_0^1\left(\log\left(\frac{1+it}{1-it}\right)\right)^{n-1}
\left(\log\frac1t\right)^{m-1}\frac{2dt}{1+t^2} \]
and hence 
\begin{align*}
&\sum_{m,n\ge1}\TT(\underbrace{1,\ldots,1}_{n-1},m)X^{m-1}Y^{n-1}\\
&=\int_0^1\left(\sum_{n\ge1}\left(\log\left(\frac{1+it}{1-it}\right)\right)^{n-1}\frac1{(n-1)!}\left(\frac{Y}{i}\right)^{n-1}\right)
\left(\sum_{m\ge1}\left(\log\frac1t\right)^{m-1}\frac{X^{m-1}}{(m-1)!}\right)\frac{2dt}{1+t^2}\\
&=\int_0^1\left(\frac{1+it}{1-it}\right)^\frac{Y}{i}t^{-X}\frac{2dt}{1+t^2}\\
&=2\int_0^1t^{-X}(1-it)^{iY-1}(1+it)^{-iY-1}dt.
\end{align*}
Recall the integral expression of the Appell hypergeometric series $F_1$ (see {\it e.g.} \cite[Chap.14 (p. 300)]{WW}):
\begin{align*}
& F_1(a;b_1,b_2;c;x,y)
=\sum_{m,n=0}^\infty \frac{(a)_{m+n}(b_1)_m(b_2)_n}{(c)_{m+n}\,m!\,n!} x^my^n\quad(|x|<1,\,|y|<1)\\
&=\frac{\Gamma(c)}{\Gamma(a)\Gamma(c-a)}
\int_0^1t^{a-1}(1-t)^{c-a-1}(1-xt)^{-b_1}(1-yt)^{-b_2}dt\quad(\Re(a)>0,\,\Re(c-a)>0).
\end{align*}
We put 
\[ a=1-X, c=a+1=2-X, b_1=1-iY, b_2=1+iY, x=i,y=-i \]
and obtain 
\begin{align*}
&\int_0^1t^{-X}(1-it)^{iY-1}(1+it)^{-iY-1}dt\\
&=\frac{\Gamma(1-X)\Gamma(1)}{\Gamma(2-X)} 
F_1(1-X;1-iY,1+iY;2-X;i,-i) \\
&=\frac1{1-X}F_1(1-X;1-iY,1+iY;2-X;i,-i).
\end{align*}
From this the result follows.
\end{proof}

\subsection{A relation of multiple $\TT$-values involving Entringer numbers}\label{subsec:Ent}

To state the next theorem, we review the `Entringer number', which counts the number of `down-up'
permutations in the symmetric group $S_{n+1}$ starting with $j+1$.
Alternatively, the Entringer numbers $\{\EE(n,j)\mid n,j\in \mathbb{Z}_{\geq 0},\ 0\leq j\leq n\}$ are
defined inductively by 
\begin{align*}
& \EE(0,0)=1,\quad \EE(n,0)=0\ \ (n>0),\\
& \EE(n,j)=\EE(n,j-1)+\EE(n-1,n-j)\ \ (n\geq 1,\ 1\leq j\leq n).
\end{align*}
Note in particular that $\EE(n,j)\in \mathbb{Z}_{\geq 0}$.  For more details, see Entringer \cite{Entringer1666} and also Stanley \cite{Stanley2010}.

\renewcommand{\arraystretch}{1.1}
\begin{table}[h]
\begin{center}
\begin{tabular} {|c|c|c|c|c|c|c|c|c|c|c|} \hline
\backslashbox{$\ \ n$}{$j\ \ $} & $ 0 $ & $ 1 $ & $ 2 $ & $ 3 $ & $ 4 $ & $ 5 $ & $ 6 $ & $ 7 $
 \\ \hline  
$0$ & $1$ &  &  &  &  &  &  &  
\\ \hline

$1$ & $0$ & $1$ &  &  &  &  &  &  
\\ \hline

$2$ & $0$ & $1$ & $1$ &  &  &  &  &  
\\ \hline

$3$ & $0$ & $1$ & $2$ & $2$ &  &  &  &  
\\ \hline

$4$ & $0$ & $2$ & $4$ & $5$ & $5$ &   &   &  
\\ \hline

$5$ & $0$ & $5$ & $10$ & $14$ & $16$ & $16$ &   &   
\\ \hline

$6$ & $0$ & $16$ & $32$ & $46$ & $56$ & $61$ & $61$  &   
\\ \hline

$7$ & $0$ & $61$ & $122$ & $178$ & $224$ & $256$ & $272$  &  $272$ 
\\ \hline
\end{tabular}
\end{center}
\label{tab:Bnk}
\caption{$\EE(n,j)  \: (0\le n,\,j \leq 7) $}
\end{table}

Incidentally, let $\EE_0=1$ and define 
\begin{equation}
\EE_n=\sum_{j=0}^{n-1}\EE(n-1,j)\  (=\EE(n,n))\quad (n\geq 1).  \label{Euler-2}
\end{equation}
Then $\EE_n$ is sometimes called the `Euler number' (see \cite{Stanley2010})
and is equal to the total number of `down-up' (or `up-down') permutations in $S_n$. 
A generating function of $\EE_n$ is
\begin{equation}
\sec x+\tan x=\sum_{n=0}^\infty \EE_n\frac{x^n}{n!}  \label{gene-Euler-2}
\end{equation}
(see \cite[Theorem 1.1]{Stanley2010}). Hence $(-1)^n\EE_{2n}=E_{2n}$ is the 
usual Euler number defined by
\begin{equation}\label{euler}
 \sum_{n=0}^\infty E_n \frac{x^n}{n!} = \frac1{\cosh x} 
\end{equation} 
(see N\"orlund \cite[Chap.\,2]{N1924}, note that $\cosh x$ is an even function) and the odd-indexed $\EE_{2n+1}$ coincides with the `tangent number' (see \cite{N1924}). 

\begin{table}[h]
\begin{center}
\begin{tabular} {|c|c|c|c|c|c|c|c|c|c|c|c|} \hline
$\ \ n$ & $ 0 $ & $ 1 $ & $ 2 $ & $ 3 $ & $ 4 $ & $ 5 $ & $ 6 $ & $ 7 $ & $ 8 $ & $ 9 $ & $10$
 \\ \hline  

$\ \ \EE_n$ & $1$ & $1$ & $1$ & $2$ & $5$ & $16$ & $61$ & $272$ & $1385$  &  $7936$ & $50521$
\\ \hline
\end{tabular}
\end{center}
\label{tab:En-9}
\caption{$\EE_n  \: (0\le n \leq 10) $}
\end{table}
Using $\EE_n$, the well-known formulas for Riemann zeta values as well as Dirichlet $L$-values of conductor 4
(sometimes referred to as the Dirichlet beta values) can be re-written as formulas for our $T$- and $\TT$-values 
in a uniform manner as 
\begin{equation}
\begin{rcases}
T(n+1) & (n:\,\text{\rm odd}) \\
\TT(n+1) & (n:\,\text{\rm even})
\end{rcases} =\frac{\EE_{n}}{n!}\left(\frac{\pi}{2}\right)^{n+1}
\label{Euler-Ttilde}
\end{equation}
for any $n\in \mathbb{Z}_{\geq 0}$ 
(see Comtet \cite{Comtet1974}).

Coming back to the Entringer number, we have the following curious and beautiful relations.  Note that
the values on the right are not appearing in \eqref{Euler-Ttilde}, {\it i.e.} the `difficult' values presumably
not rational multiples of powers of $\pi$. 

\begin{theorem}\label{Ent-rel}\ \ For any $n\in \mathbb{Z}_{\geq 1}$, we have
\begin{align}
& \sum_{j=1}^{n}\EE(n,j)\TT(\underbrace{1,\ldots,1,\overset{\tiny j}{\check{2}},1,\ldots,1}_{n}) =
\begin{cases}
\widetilde{T}(n+1) & (n:\,\text{\rm odd}),\\
T(n+1) & (n:\,\text{\rm even}).
\end{cases}
\label{Entringer-Ttilde}
\end{align}
\end{theorem}

\begin{example}
Examples in low weights are
\begin{align*}
T(3)&=\TT(2,1)+\TT(1,2),\\
\TT(4)&=\TT(2,1,1)+2\TT(1,2,1)+2\TT(1,1,2),\\
T(5)&=2\TT(2,1,1,1)+4\TT(1,2,1,1)+5\TT(1,1,2,1)+5\TT(1,1,1,2),\\
\TT(6)&=5\TT(2,1,1,1,1)+10\TT(1,2,1,1,1)+14\TT(1,1,2,1,1)\\
& \quad +16\TT(1,1,1,2,1)+16\TT(1,1,1,1,2),\\
T(7)&=16\TT(2,1,1,1,1,1)+32\TT(1,2,1,1,1,1)+46\TT(1,1,2,1,1,1)\\
  & \quad +56\TT(1,1,1,2,1,1)+61\TT(1,1,1,1,2,1)+61\TT(1,1,1,1,1,2).
\end{align*}
\end{example}

\begin{proof}
We give a proof by induction on $n\geq 1$. The case $n=1$ becomes the trivial identity $\TT(2)=\TT(2)$. 
Let $k\geq 2$ and assume that the assertions for $n\leq k-1$ hold and consider the case $n=k$. 
From the identity~\eqref{key}, by taking out the term with $j=k$ on the right-hand side and noting $(k+1)\TT(\underbrace{1,\ldots,1}_{k+1})
=\TT(1)\TT(\underbrace{1,\ldots,1}_{k})$ corresponds to the term $j=0$, we have
\begin{align*}
& i^{k} T(k+1) -i^{k+1} \TT(k+1) =\sum_{j=0}^{k-1}i^{j+1}\TT(j+1)\TT(\underbrace{1,\ldots,1}_{k-j})+\TT(\underbrace{1,\ldots,1}_{k-1},2).
\end{align*}
We use the real part of this identity:
\begin{align}
& {\rm Re}\left(i^{k} T(k+1) -i^{k+1} \TT(k+1)\right) =\sum_{j=0 \atop \textrm{$j$\,:\,odd}}^{k-1}(-1)^{(j+1)/2}\,\TT(j+1)\TT(\underbrace{1,\ldots,1}_{k-j})+\TT(\underbrace{1,\ldots,1}_{k-1},2). \label{eq-5-1-1}
\end{align}
By the induction hypothesis, we can substitute 
$$\widetilde{T}(j+1)=\sum_{h=1}^{j}\EE(j,h)\TT(\underbrace{1,\ldots,1,\overset{\tiny h}{\check{2}},1,\ldots,1}_{j})\quad (\textrm{$j$\,:\,odd},\ 1\leq j\leq k-1)$$
into \eqref{eq-5-1-1} and obtain
\begin{align*}
& {\rm Re}\left(i^{k} T(k+1) -i^{k+1} \TT(k+1)\right) \notag \\
& \ \ =\sum_{j=1 \atop \textrm{$j$\,:\,odd}}^{k-1}(-1)^{(j+1)/2}\sum_{h=1}^{j}\EE(j,h)\TT(\underbrace{1,\ldots,1,\overset{\tiny h}{\check{2}},1,\ldots,1}_{j})
\TT(\underbrace{1,\ldots,1}_{k-j})+\TT(\underbrace{1,\ldots,1}_{k-1},2). 
\end{align*}
Since multiple $\TT$-values satisfy the same shuffle product formulas as the usual multiple zeta values, we have the following  relation:
\begin{align*}
& \TT(\underbrace{1,\ldots,1,\overset{\tiny h}{\check{2}},1,\ldots,1}_{m})\TT(\underbrace{1,\ldots,1}_{n}) =\sum_{l=h}^{n+h}\binom{l}{h}\binom{m+n-l}{m-h}\TT(\underbrace{1,\ldots,1,\overset{\tiny l}{\check{2}},1,\ldots,1}_{m+n}).
\end{align*}
Therefore we obtain
\begin{align*}
& {\rm Re}\left(i^{k} T(k+1) -i^{k+1} \TT(k+1)\right)-\TT(\underbrace{1,\ldots,1}_{k-1},2) \notag \\
& \ \ =\sum_{j=1 \atop \textrm{$j$\,:\,odd}}^{k-1}(-1)^{(j+1)/2}\sum_{h=1}^{j}\EE(j,h)\sum_{l=h}^{k-j+h}\binom{l}{h}\binom{k-l}{j-h}\TT(\underbrace{1,\ldots,1,\overset{\tiny l}{\check{2}},1,\ldots,1}_{k})\\
& \ \ =\sum_{l=1}^{k}\left(\sum_{j=1 \atop \textrm{$j$\,:\,odd}}^{k-1}(-1)^{(j+1)/2}\sum_{h=1}^{j}\EE(j,h)\binom{l}{h}\binom{k-l}{j-h}\right)\TT(\underbrace{1,\ldots,1,\overset{\tiny l}{\check{2}},1,\ldots,1}_{k})\\
& \ \ =\sum_{l=1}^{k}\sum_{h=1}^{k-1}\sum_{j=h \atop \textrm{$j$\,:\,odd}}^{k-1}(-1)^{(j+1)/2}\binom{l}{h}\binom{k-l}{j-h}\EE(j,h)\TT(\underbrace{1,\ldots,1,\overset{\tiny l}{\check{2}},1,\ldots,1}_{k})\\
& \ \ =\sum_{l=1}^{k}\sum_{h=1}^{l}\sum_{j=h \atop \textrm{$j$\,:\,odd}}^{k}(-1)^{(j+1)/2}\binom{l}{h}\binom{k-l}{j-h}\EE(j,h)\TT(\underbrace{1,\ldots,1,\overset{\tiny l}{\check{2}},1,\ldots,1}_{k})\\
& \qquad -\delta_{k,\textrm{odd}}\sum_{l=1}^{k}(-1)^{(k+1)/2}\EE(k,l)\TT(\underbrace{1,\ldots,1,\overset{\tiny l}{\check{2}},1,\ldots,1}_{k}),
\end{align*}
where $\delta_{k,\textrm{odd}}=1\ ({\rm resp.}\ 0)$ if $k$ is odd (resp. even). 
By replacing $j$ with $h+j$, the last expression becomes 
\begin{align*}
& \sum_{l=1}^{k}\sum_{h=1}^{l}\sum_{j=0 \atop \textrm{$h+j$\,:\,odd}}^{k-l}(-1)^{(h+j+1)/2}\binom{l}{h}\binom{k-l}{j}\EE(h+j,h)\TT(\underbrace{1,\ldots,1,\overset{\tiny l}{\check{2}},1,\ldots,1}_{k})\\
& \qquad -\delta_{k,\textrm{odd}}\sum_{l=1}^{k}(-1)^{(k+1)/2}\EE(k,l)\TT(\underbrace{1,\ldots,1,\overset{\tiny l}{\check{2}},1,\ldots,1}_{k}).
\end{align*}
Here, we need
\begin{lemma}\label{Lem-5-2}\ \ For $k,l\in \mathbb{Z}_{\geq 0}$, 
\begin{align}
& \sum_{h=0}^{l}\sum_{j=0}^{k-l}\,i^{h+j}\binom{l}{h}\binom{k-l}{j}\EE(h+j,h)   =i^{k-1}\EE(k,l)+\delta_{k,l}\,i+\delta_{l,0}, \label{eq-Lem-5-2}
\end{align}
where $\delta_{k,l}$ is the Kronecker delta.
\end{lemma}
Multiplying both sides of \eqref{eq-Lem-5-2} by $i$ and taking the real part, we have
$$\sum_{h=1}^{l}\sum_{j=0 \atop \textrm{$h+j$\,:\,odd}}^{k-l}(-1)^{(h+j+1)/2}\binom{l}{h}\binom{k-l}{j}\EE(h+j,h) =\delta_{k,\textrm{even}}(-1)^{k/2}\EE(k,l)-\delta_{k,l},$$
where $\delta_{k,\textrm{even}}=1\ ({\rm resp.}\ 0)$ if $k$ is even (resp. odd). 
Therefore we obtain 
\begin{align*}
& {\rm Re}\left(i^{k} T(k+1) -i^{k+1} \TT(k+1)\right) \notag \\
& \ \ =\sum_{l=1}^{k}\left(\delta_{k,\textrm{even}}(-1)^{k/2}\EE(k,l)-\delta_{k,l}\right)\TT(\underbrace{1,\ldots,1,\overset{\tiny l}{\check{2}},1,\ldots,1}_{k}) +\TT(\underbrace{1,\ldots,1}_{k-1},2)\\
& \qquad -\delta_{k,\textrm{odd}}\sum_{l=1}^{k}(-1)^{(k+1)/2}\EE(k,l)\TT(\underbrace{1,\ldots,1,\overset{\tiny l}{\check{2}},1,\ldots,1}_{k})\\
& \ \ =\delta_{k,\textrm{even}}(-1)^{k/2}\sum_{l=1}^{k}\EE(k,l)\TT(\underbrace{1,\ldots,1,\overset{\tiny l}{\check{2}},1,\ldots,1}_{k})\\
& \qquad -\delta_{k,\textrm{odd}}(-1)^{(k+1)/2}\sum_{l=1}^{k}\EE(k,l)\TT(\underbrace{1,\ldots,1,\overset{\tiny l}{\check{2}},1,\ldots,1}_{k}).
\end{align*}
This gives the desired identities when $n=k$ for both even and odd $k$.
\end{proof}

\begin{proof}[Proof of Lemma \ref{Lem-5-2}] 
We use the following generating function of the Entringer numbers (see \cite[(2.2)]{Stanley2010}):
\begin{align}
&\sum_{k=0}^\infty\sum_{l=0}^{k}\,\EE(k,l)\frac{x^{k-l}}{(k-l)!}\frac{y^l}{l!}=\frac{\cos x+\sin y}{\cos (x+y)}. \label{generating E-kl}
\end{align}
Let $I_1(x,y)$ be the generating function of the left-hand side of \eqref{eq-Lem-5-2}:
\begin{align*}
I_1(x,y) & =\sum_{k=0}^\infty \sum_{l=0}^{k} \sum_{h=0}^l \sum_{j=0}^{k-l} i^{h+j}\binom{l}{h}\binom{k-l}{j}\EE(h+j,h)\frac{x^{k-l}}{(k-l)!}\frac{y^l}{l!}.
\end{align*}
By \eqref{generating E-kl}, we have
\begin{align*}
I_1(x,y) & =\sum_{h=0}^\infty \sum_{j=0}^\infty  i^{h+j} \EE(h+j,h) \sum_{l=0}^\infty\binom{l}{h} \frac{y^l}{l!}\sum_{k=l}^{\infty}\binom{k-l}{j}\frac{x^{k-l}}{(k-l)!}\\
 & =\sum_{h=0}^\infty \sum_{j=0}^\infty  i^{h+j} \EE(h+j,h) \sum_{l=0}^\infty\binom{l}{h} \frac{y^l}{l!}\sum_{k=0}^{\infty}\binom{k}{j}\frac{x^{k}}{k!}\\
 & =\sum_{h=0}^\infty \sum_{j=0}^\infty  i^{h+j} \EE(h+j,h) \frac{x^j}{j!}\frac{y^{h}}{h!}e^{x+y}\\
 & =\sum_{m=0}^\infty \sum_{h=0}^{m} i^{m} \EE(m,h) \frac{x^{m-h}}{(m-h)!}\frac{y^{h}}{h!}e^{x+y}\\
 & =\frac{\cos(ix)+\sin(iy)}{\cos(i(x+y))}e^{x+y}\\
 & =\frac{e^x+e^{-x}+ie^y-ie^{-y}}{1+e^{-2x-2y}}.
\end{align*}
On the other hand, the generating function of the right-hand side of \eqref{eq-Lem-5-2} is 
\begin{align*}
I_2(x,y) & =\sum_{k=0}^\infty \sum_{l=0}^{k} \left(i^{k-1}\EE(k,l)+\delta_{k,l}i+\delta_{l,0}\right)\frac{x^{k-l}}{(k-l)!}\frac{y^l}{l!}\\
& =-i\sum_{k=0}^\infty \sum_{l=0}^{k} \EE(k,l)\frac{(ix)^{k-l}}{(k-l)!}\frac{(iy)^l}{l!}+ie^y+e^x\\
& =-i\left(\frac{e^x+e^{-x}+ie^{y}-ie^{-y}}{e^{x+y}+e^{-x-y}}\right)+ie^y+e^x\\
& =\frac{e^x+e^{-x}+ie^y-ie^{-y}}{1+e^{-2x-2y}}.
\end{align*}
Hence $I_1(x,y)=I_2(x,y)$, which gives the proof of \eqref{eq-Lem-5-2}, and now the proof of Theorem~\ref{Ent-rel} is complete.
\end{proof}

\subsection{Relation to modular forms}\label{subsec:modular}

In this subsection, we present our experimental discovery on a possible connection between double $\TT$- 
(as well as $T$-) values and modular forms
of level 4 (and 2), or precisely speaking, connection to some `period polynomials' associated to those modular forms.
This is certainly an analogous phenomenon to our previously studied relations between double zeta values and modular forms on the full modular group (\cite{GKZ}).
We however could not give a proof and leave it to the interested readers. 

For $N=2$ and $4$, even integer $k\ge4$, and $1\le j\le (k-2)/2$, define the polynomial $\widetilde S_{N,k,j}(X)$ with rational coefficients by
\begin{align*}
\widetilde S_{N,k,j}(X)&=\frac{N^{k-2j-1}}{k-2j} X^{k-2} B_{k-2j}^0\left(\frac1{NX}\right)-\frac1{2j}B_{2j}^0(X)\\
&\qquad \quad -\frac{kB_{2j}B_{k-2j}}{2j(k-2j)B_k}\left(\frac{1-2^{-2j}}{1-2^{-k}}\frac{X^{k-2}}{N}-\frac{1-2^{-k+2j}}{1-2^{-k}}\frac1{N^{2j}}\right),
\end{align*}
where $B_{n}$ is the Bernoulli number and $B_{n}^0(X)$ is the usual Bernoulli polynomial with the term $nB_1X^{n-1}$ removed:
\[  B_{n}^0(X) =\sum_{0\le j\le n \atop \text{$j$:even}}\binom{n}{j} B_j X^{n-j}. \]
We further define $P_{N,k,j}(X)$ and $P_{N,k,j}^{(\pm)}(X)$  by using $\widetilde S_{N,k,j}(X)$ as
\[ P_{N,k,j}(X)   
=(-2X+2)^{k-2}\widetilde S_{N,k,j}\left(\frac{X+1}{-2X+2}\right) \]
and 
\[ P_{N,k,j}^{(\pm)}(X)=\frac12\left(P_{N,k,j}(X)\pm P_{N,k,j}(-X)\right).  \]
We can now state our conjecture.
 
\begin{conj}\label{conj-modular}
1)  For $N=2$ or $4$, even integers $k\ge4$, and integers $j$ with $1\le j\le (k-2)/2$, write the polynomial $P_{N,k,j}^{(+)}(X+1)$ as
\[P_{N,k,j}^{(+)}(X+1) = \sum_{i=0}^{k-2} a_i \binom{k-2}{i} X^i.  \]
(Each coefficient $a_i$ depends on $N$, $k$, and $j$.)
Then we have the following relation among the double $\TT$-values:
\[ \sum_{i=0}^{k-2} a_i\, \TT(i+1,k-i-1)=0. \]

The $\Q$-vector space $V_{4,k}$ spanned by $ P_{4,k,j}^{(+)}(X)\ (1\le j \le (k-2)/2)$ is of dimension $\left[(k-2)/4\right]$, which we conjecture
to be equal to the number of independent relations among double $\TT$-values of weight $k$. The polynomials $P_{2,k,j}^{(+)}(X)$ are contained in
$V_{4,k}$, and span the subspace of dimension $\left[(k-2)/6\right]$.

2)  For $N=2$ or $4$, even integers $k\ge4$, and integers $j$ with $1\le j\le (k-2)/2$, write the polynomial $P_{N,k,j}^{(-)}(X+1)$ as
\[P_{N,k,j}^{(-)}(X+1) = \sum_{i=0}^{k-3} b_i \binom{k-2}{i} X^i.  \]
(Here too we suppress the dependence on $N$ etc. in the notation of $b_i$. Note that the degree of the odd polynomial $P_{N,k,j}^{(-)}(X)$ is at most $k-3$.)
Then we have the following relation among the double $T$-values:
\[ \sum_{i=0}^{k-3} b_i\, T(i+1,k-i-1)=0. \]

In this case, the $\Q$-vector space $W_k$ spanned by $P_{4,k,j}^{(-)}(X)\ (1\le j \le (k-2)/2)$ is the same as that spanned by $P_{2,k,j}^{(-)}(X)\ (1\le j \le (k-2)/2)$,
and the conjectural dimension of $W_k$ is $\left[k/4\right]-1$.  The conjectural number of independent relations among double $T$-values is  $k/2-2$.

\end{conj}

\begin{remark}
i)  The polynomial $\widetilde S_{N,k,j}(X)$ is the period polynomial $r^{+}(R_{\Gamma_0(N),k-2,2j-1})(X)$ in the work of Fukuhara and Yang \cite{FY}.
A period polynomial is a polynomial associated to a cusp form whose coefficients are `periods' of the given cusp form.
They computed period polynomials explicitly for some specific cusp  forms $R_{\Gamma_0(N),k-2,2j-1}$ on the congruence subgroup $\Gamma_0(N)$, and exhibited several 
properties of those. For more details, see their paper \cite{FY} and the references therein. For basics of modular forms, see e.g. Miyake~\cite{Miyake}.

ii) The number $\left[(k-2)/4\right]$ appeared in 1) of the above conjecture is equal to the difference $\dim S_k(\Gamma_0(4))-\dim S_k(\Gamma_0(2))$ of dimensions of the spaces of 
cusp forms of weight $k$ on the congruence subgroups $\Gamma_0(4)$ and $\Gamma_0(2)$. 
And the difference $\left[(k-2)/4\right]-\left[(k-2)/6\right]$ is equal to the dimension of the space of new forms of weight $k$ on $\Gamma_0(4)$. Also, 
the number $\left[k/4\right]-1$ in 2) is equal to $\dim S_k(\Gamma_0(2))$, whereas  $k/2-2=\dim S_k(\Gamma_0(4))$.
We could not find how to produce the remaining $k/2-2-(\left[k/4\right]-1)=\left[(k-2)/4\right]$ relations of double $T$-values
via a similar procedure.

iii) By Fukuhara and Yang~\cite[Cor.~1.9]{FY2} (resp. \cite[Cor.~1.5]{FY}), the polynomials $\widetilde S_{4,k,j}(X)$ (resp. $\widetilde S_{2,k,j}(X)$)  
span the $k/2-2=\dim S_k(\Gamma_0(4))$ (resp. $\left[k/4\right]-1=\dim S_k(\Gamma_0(2))$) dimensional space.

\end{remark}

\begin{example}
i)  For $N=4$, $k=6$, and $j=1$, we have 
\[ \widetilde S_ {4,6,1}(X)=-\frac12 X^4+\frac12 X^2-\frac1{32},\quad  P_{4,6,1}^{(+)}(X)=X^4-10X^2+1\]
and
\begin{align*} P_{4,6,1}^{(+)}(X+1) &= - 8 - 16X - 4X^2+ 4X^3 +X^4\\
& =-8\binom{4}{0}-4\binom{4}{1}X -\frac{2}{3} \binom{4}{2} X^2+1\cdot\binom{4}{3}X^3+1\cdot \binom{4}{4}X^4\\
&=-\frac13\left(24\binom{4}{0}+12\binom{4}{1}X +2\binom{4}{2} X^2-3\binom{4}{3}X^3-3 \binom{4}{4}X^4\right).  
\end{align*}
Accordingly, we numerically (to high precision) have
\[  24\TT(1,5)+12\TT(2,4)+2\TT(3,3)-3\TT(4,2)-3\TT(5,1)=0.\]

ii)  For $N=2$, $k=8$, and $j=2$, we have 
\begin{align*} \widetilde S_ {2,8,2}(X)&=-\frac1{17} X^6+\frac14X^4-\frac18 X^2+\frac1{136},\\ 
P_{2,8,2}^{(+)}(X)&=-\frac{2}{17}(5X^6-61X^4-61X^2+5)
\end{align*}
and
\begin{align*} P_{2,8,2}^{(+)}(X+1) &=\frac2{17}(112+ 336X+ 352X^2+144X^3- 14X^4 - 30X^5-5 X^6 )\\
&=\frac2{17}\left(112\binom{6}{0}+56\binom{6}{1}X+\frac{352}{15}\binom{6}{2}X^2+\frac{36}{5}\binom{6}{3}X^3\right.\\
&\qquad\qquad  \left. -\frac{14}{15}\binom{6}{4}X^4-5\binom{6}{5}X^5-5\binom{6}{6}X^6\right).
\end{align*}
We compute 
\[ 112\TT(1,7)+56\TT(2,6)+\frac{352}{15}\TT(3,5)+\frac{36}{5}\TT(4,4)-\frac{14}{15}\TT(5,3)-5\TT(6,2)-5\TT(7,1)=0 \]
to very high precision.

iii)  For $N=4$, $k=8$, and $j=1$, we have 
\[\widetilde S_ {4,8,1}(X)=\frac{208}{51} X^6-\frac{16}3 X^4 +\frac76 X^2-\frac{19}{408},\quad
P_{4,8,1}^{(-)}(X)=  -\frac{640}{17}\left(X^5 -8 X^3 + X\right) \]
and
\begin{align*} &P_{4,8,1}^{(-)}(X+1) = \frac{640}{17}\left(6+18X+14X^2-2X^3-5X^4-X^5\right)  \\
&=\frac{64}{51}\left(180\binom{6}{0}+90\binom{6}{1}X +28\binom{6}{2} X^2-3\binom{6}{3}X^3-10 \binom{6}{4}X^4-5\binom{6}{5} X^5\right).   
\end{align*}
The corresponding conjectural relation is  
\[  180T(1,7)+90 T(2,6)+28 T(3,5)-3 T(4,4)-10 T(5,3)-5 T(6,2)=0.\]
\end{example}

At the end of this section, we mention a connection to our previously obtained `weighted sum formula'.

If we start with the polynomials $X^{k-2}$ or $1$ instead of $\widetilde S_{N,k,j}(X)$, we obtain the weighted sum formulas
for double $\TT$- and $T$- values. More precisely, for even $k\ge 4$, take the even and odd parts of the polynomial
\[ (-2X+2)^{k-2}\left(\frac{X+1}{-2X+2}\right)^{k-2} = (X+1)^{k-2},\]
namely $\left((X+1)^{k-2}\pm (-X+1)^{k-2}\right)/2$ and make the shift $X\to X+1$ to obtain
\[ \frac12 \left((X+2)^{k-2}\pm (-X)^{k-2}\right) =\frac12 \left( \sum_{j=0}^{k-2} 2^{k-j}\binom{k-2}{j} X^j\pm X^{k-2}\right) .\]
Then, this `corresponds' to the weighted sum formula
\[ \sum_{j=0}^{k-2} 2^{k-j-2}\TT(j+1,k-1-j)+\TT(k-1,1)=(k-1)T(k) \]
proved in \cite[Prop.~4.2]{AK2004} in the `$+$' case, and 
\[ \sum_{j=0}^{k-3} 2^{k-j-2}T(j+1,k-1-j)=(k-1)T(k) \]
proved in \cite[Th.~3.2]{KT2020-Tsukuba} in the `$-$' case.  If we start with $1$ instead, the resulting polynomial is essentially the same.

\section{Certain zeta functions and poly-Bernoulli and Euler numbers}\label{sec:zeta-polyEuler}

In our previous work \cite{AK1999, KT2020-ASPM}, we studied zeta functions
\begin{align}
\xi(k_1,\ldots,k_r;s)&=\frac{1}{\Gamma(s)}\int_0^\infty {t^{s-1}}\frac{\Li (k_1,\ldots,k_r;1-e^{-t})}{e^t-1}\,dt \label{Def-xi} \\
\intertext{and}
\psi(k_1,\ldots,k_r;s)&=\frac{1}{\Gamma(s)}\int_0^\infty {t^{s-1}}\frac{A(k_1,\ldots,k_r;\tanh(t/2))}{\sinh t}\,dt,
\label{Def-psi}
\end{align}
both converge in ${\rm Re}(s)>0$ and are analytically continued to entire functions.
Here,  
\[ \Li(k_1,\ldots,k_r;z)=\sum_{0< m_1<\cdots<m_r} \frac{z^{m_r}}{m_1^{k_1}\cdots m_r^{k_r}}\qquad (k_1,\ldots,k_r\in \mathbb{Z};\ |z|<1) \]
is the multiple polylogarithm and $A(k_1,\ldots,k_r;z)$ is its `level 2' analogue already appeared in \eqref{Def-A-polylog}.

We now introduce a level 4 analogue of these zeta functions by using the level 4 variant $\aa(k_1,\ldots,k_r;z)$ of multiple 
polylogarithm, which is defined by  the iterated integral as follows.  The idea is just to replace the starting point $0$ 
of the iterated integral of $A(k_1,\ldots,k_r;z)$ with $i=\sqrt{-1}$.

\begin{definition} \label{Def-2-7}
For $k_1,\ldots,k_r\in \mathbb{Z}_{\geq 1}$, define
\begin{align}\label{Def-aa-multi}
\aa(k_1,\ldots,k_r;z)
& =
\begin{cases}
\displaystyle{\int_{i}^{z}\frac{1}{u}\,{\aa(k_1,\ldots,k_{r-1},k_r-1;u)}\,du} & (k_r\geq 2),\\
\displaystyle{\int_{i}^{z}\frac{2}{1-u^2}\,\aa(k_1,\ldots,k_{r-1};u)\,du} & (k_r=1),
\end{cases}
\end{align}
with $\aa(\emptyset;u)=1$. In particular, 
\begin{equation}\label{aa1} \aa(1;z)=\int_{i}^{z}\frac{2}{1-u^2}=A(1;z)-A(1;i)=2\tanh^{-1}(z)-\frac{\pi i}{2}. \end{equation}
\end{definition}

Note that if $k_r\ge2$, then we may consider the value of $\aa(k_1,\ldots,k_r;z)$ at $z=1$, and in fact this was already
appeared in the proof of Theorem~\ref{Teven}. We state the formula again as a proposition.

\begin{prop} For $k_1,\ldots,k_r\in\Z_{\ge1}$ with $k_r\ge2$, let $(l_1,\ldots,l_s)$ be the dual index of 
$(k_1,\ldots,k_r)$. Then we have
\begin{equation}\label{aaat1}
\aa(k_1,\ldots,k_r;1)=i^{r-k}\TT(l_1,\ldots,l_s). 
\end{equation}
\end{prop}

\begin{proof} As already done in the proof of Theorem~\ref{Teven}, this can be shown by the change of variables $t\to (-iu+1)/(u-i)$
in the iterated integral 
\begin{align*}
&\aa(k_1,\ldots,k_r;1)\\
&=\int_i^1\frac{2dt}{1-t^2}\circ\underbrace{\frac{dt}{t}\circ\cdots\circ\frac{dt}{t}}_{k_1-1}\circ
\frac{2dt}{1-t^2}\circ\underbrace{\frac{dt}{t}\circ\cdots\circ\frac{dt}{t}}_{k_2-1}\circ\cdots\cdots\circ
\frac{2dt}{1-t^2}\circ\underbrace{\frac{dt}{t}\circ
\cdots\circ\cdots\circ\frac{dt}{t}}_{k_r-1}
\end{align*}
and by the formula \eqref{integ-exp2}.
\end{proof}

We also record here a formula of $A(\underbrace{1,\ldots,1}_{r-1},k+1;z)$ expressed in terms of $\TT$-values, $\aa(\underbrace{1,\ldots,1}_{j-1},k+1;z)\ (1\le j\le r)$, and $\log z$.  
Specialization $z=1$ gives  \eqref{TbyTtilde}.

\begin{prop} For $r,k\ge1$, we have
\begin{align*}
&A(\underbrace{1,\ldots,1}_{r-1},k+1;z)\\
&=\sum_{j=1}^r i^{r-j}\TT(\underbrace{1,\ldots,1}_{r-j})\aa(\underbrace{1,\ldots,1}_{j-1},k+1;z)
+\sum_{j=0}^k i^r \TT(\underbrace{1,\ldots,1}_{r-1},j+1)\frac{(\log z-\log i)^{k-j}}{(k-j)!}.
\end{align*}
\end{prop}

\begin{proof}  This can be shown in the same manner using the path composition formula
as in the calculation of \eqref{TbyTtilde}.  We omit the detail here.
\end{proof}

Now we define the zeta function associated to our $\aa$.

\begin{definition}\label{Def-3-8}\ 
For $k_1,\ldots,k_r\in \mathbb{Z}_{\geq 1}$, set
\begin{equation}
\lambda(k_1,\ldots,k_r;s)=\frac{1}{\Gamma(s)}\int_0^\infty {t^{s-1}}\frac{\aa (k_1,\ldots,k_r;\tanh(t/2+\pi i/4))}{\cosh t}\,dt\quad ({\rm Re}(s)>0). \label{Def-lambda}
\end{equation}
\end{definition}

Before discussing the convergence of the integral in ${\rm Re}(s)>0$, let us first explain our motivation of introducing these functions, 
in other words, explain how we chose the integrants of these.   For this, we recall `poly-Bernoulli numbers'.

Poly-Bernoulli numbers, having two versions $B_n^{(k)}$ and $C_n^{(k)}$,  were defined by 
the first named author in \cite{Kaneko1997} and 
in Arakawa-Kaneko \cite{AK1999} by using generating series.  For an integer $k\in \mathbb{Z}$, 
the sequences $\{B_n^{(k)}\}$ and $\{C_n^{(k)}\}$ of rational numbers are given by 
\begin{align*}
&\frac{{\rm Li}_{k}(1-e^{-t})}{1-e^{-t}}=\sum_{n=0}^\infty B_n^{(k)}\frac{t^n}{n!},\qquad \frac{{\rm Li}_{k}(1-e^{-t})}{e^t-1}=\sum_{n=0}^\infty C_n^{(k)}\frac{t^n}{n!},  
\end{align*}
where ${\rm Li}_{k}(z)={\rm Li}(k;z)$ (in the notation above) is the classical polylogarithm function (or rational function when $k\le0$) given by
\begin{equation}
{\rm Li}_{k}(z)=\sum_{m=1}^\infty \frac{z^m}{m^k}\quad (|z|<1). \label{Def-polylog}
\end{equation}
Since ${\rm Li}_1(z)=-\log(1-z)$, we see that
$B_n^{(1)}$ and $C_n^{(1)}$
are usual Bernoulli numbers, where the only difference being $B_n^{(1)}=1/2$ and $C_n^{(1)}=-1/2$. 
A multiple version  $C_n^{(k_1,\ldots,k_r)}$ of $C_n^{(k)}$ is the multi-poly-Bernoulli numbers defined in Imatomi-Kaneko-Takeda \cite{IKT2014} by the generating series
\begin{equation}\label{mpBern}
\frac{{\rm Li}(k_1,\ldots,k_r;1-e^{-t})}{e^t-1}=\sum_{n=0}^\infty C_n^{(k_1,\ldots,k_r)}\frac{t^n}{n!}.
\end{equation}
The function $\xi(k_1,\ldots,k_r;s)$ was introduced as the one interpolating $\{C_n^{(k_1,\ldots,k_r)}\}$ at negative integer arguments:
$$\xi(k_1,\ldots,k_r;-n)=(-1)^n C_n^{(k_1,\ldots,k_r)}\quad (n\in \mathbb{Z}_{\geq 0}).$$
Note that the left-hand side of the definition \eqref{mpBern} is exactly the same as the function appearing in the integral of the definition
\eqref{Def-xi} of $\xi(k_1,\ldots,k_r;s)$, and we observe that this function can be realized as 
\[ \frac{{\rm Li}(k_1,\ldots,k_r;1-e^{-t})}{e^t-1}=\frac{d}{dt}{\rm Li}(k_1,\ldots,k_{r-1},k_r+1;1-e^{-t}) \]
and moreover that the function $1-e^{-t}$ is the inverse of ${\rm Li}(1;t)=-\log(1-t)$:
\[ {\rm Li}(1;1-e^{-t})=t. \]
Likewise, we see that the function appearing in \eqref{Def-psi} is
\[ \frac{A(k_1,\ldots,k_r;\tanh(t/2))}{\sinh t}=\frac{d}{dt}A(k_1,\ldots,k_{r-1},k_r+1;\tanh(t/2)), \]
where $\tanh(t/2)$ is the inverse of $A(1;t)$ (see \eqref{A1exp}).

Our newly introduced function $\aa(k_1,\ldots,k_r; z)$ satisfies, by definition~\eqref{Def-aa-multi}, the same derivative formula as~\eqref{deriv-a-multi}
and therefore, if we follow the same line, it would be natural to consider the function
\[ \frac{d}{dt}\aa(k_1,\ldots,k_{r-1},k_r+1;h(t)), \]
where $h(t)$ is the inverse of $\aa(1;t)$.  

\begin{lemma}\label{defofh} The function
\begin{equation}
h(x)=\tanh (x)+\frac{i}{\cosh (x)},  \label{Def-tau}
\end{equation}
which is also equal to 
\begin{equation}
\tanh\left(\frac{x}{2}+\frac{\pi i}{4}\right), \label{tau-exp1} 
\end{equation}
is the inverse of $\aa(1;x)$.  We have 
\begin{equation*}
\frac{d}{dx}h(x)=-\frac{i}{\cosh x}h(x),  \label{tau-deri} 
\end{equation*}
and thus for $k_1,\ldots,k_r\in \mathbb{Z}_{\geq 1}$, 
\begin{equation}
\frac{d}{dx}\aa (k_1,\ldots,k_r;h(x))=
\begin{cases}
-i\,\dfrac{\aa (k_1,\ldots,k_r-1;h(x))}{\cosh x} & (k_r\geq 2)\\
-i\aa(k_1,\ldots,k_{r-1};h(x)) & (k_r=1).
\end{cases}
\label{deriv-aa}
\end{equation}

\end{lemma}
\begin{proof}
We know from \eqref{aa1} that
\[ \aa(1;x)=2\tanh^{-1}(x)-\frac{\pi i}{2}, \]
and so $\tanh\left(x/2+\pi i/4\right)$ is the inverse of $\aa(1;x)$:
\[ \aa(1;\tanh\left(\frac{x}{2}+\frac{\pi i}{4}\right))=2\tanh^{-1}(\tanh\left(\frac{x}{2}+\frac{\pi i}{4}\right))-\frac{\pi i}{2}=x.\]
Using the duplication formulas, we compute
\begin{align*}
\tanh (x)+\frac{i}{\cosh (x)}&=\frac{2\sinh (x/2) \,\cosh (x/2) +i(\cosh^2 (x/2)-\sinh^2 (x/2))}{\cosh^2(x/2) +\sinh^2 (x/2)}\\
& =\frac{i(\cosh (x/2)-i\sinh (x/2))^2}{(\cosh (x/2) +i\sinh (x/2))(\cosh (x/2) -i\sinh (x/2))}\\
& =\frac{(1+i)e^{x/2} -(1-i)e^{-{x/2}}}{(1+i)e^{x/2} +(1-i)e^{-{x/2}}}=\frac{e^{{x/2}+\pi i/4} -e^{-{x/2}-\pi i/4}}{e^{{x/2}+\pi i/4} +e^{-{x/2}-\pi i/4}}\\
&=\tanh\left(\frac{x}{2}+\frac{\pi i}{4}\right).
\end{align*}
The  derivative formula follows from the definition \eqref{Def-aa-multi}.
\end{proof}

It is amusing to note that this
$h(x)$ is essentially equal to the generating function~\eqref{gene-Euler-2} of the Euler numbers $\{\EE_n\}$ appeared in \S\ref{subsec:Ent}:
$$\frac{1}{i}h(ix)=\sec x+ \tan x.$$

By using \eqref{deriv-aa}, we may obtain the following estimate and the convergence of the integral \eqref{Def-lambda} in ${\rm Re}(s)>0$ follows.

\begin{lemma}\label{Lem-3-5} 
For $k_1,\ldots,k_r\in \mathbb{Z}_{\geq 1}$, 
\begin{equation}
\aa (k_1,\ldots,k_r;h(x))=O\left(x^{k_1+\cdots+k_r}\right)\quad (x\to \infty). \label{Order-aa}
\end{equation}
\end{lemma}

\begin{proof}
We proceed by double induction on $r$ and $k_r$ . 

When $r=1$ and  $k_1=1$, the left-hand side is equal to $x$ and the assertion is obvious. 
For $k_1\geq 2$, by \eqref{deriv-aa} and $\cosh x \to \infty$ $(x\to \infty)$, there exists $C>0$ such that
$$|\aa (k_1;h(x))|\leq \int_0^{x}\frac{1}{|\cosh u|}|\aa (k_1-1;h(u))|\,du\leq C\int_0^{x}u^{k_1-1}\,du=\frac{C}{k_1}x^{k_1}$$
for sufficiently large $x>0$. Thus by induction we obtain the assertion when $r=1$.

Consider the case $r\geq 2$. When $k_r=1$, then we immediately obtain the assertion from \eqref{deriv-aa} and the
case $r-1$ because
\[ \aa (k_1,\ldots,k_r;h(x))=-i \int_0^x \aa (k_1,\ldots,k_{r-1};h(u))\,du. \]
Note that $h(0)=i$ and $\aa (k_1,\ldots,k_r;h(0))=0$.
If $k_r\geq 2$, then, again by \eqref{deriv-aa} we may argue similarly as in the case of $r=1$ and we obtain the assertion 
by induction.
\end{proof}

\begin{example}   By the standard integral representation of $L(s,\chi_4)$,
\[ L(s,\chi_4)=\frac1{2\Gamma(s)}\int_0^\infty \frac{t^{s-1}}{\cosh t}\, dt\quad ({\rm Re}(s)>0),  \]
we have 
$$\lambda(1;s)=\frac1{\Gamma(s)}\int_0^\infty \frac{t^{s}}{\cosh t}\, dt=2sL(s+1,\chi_4).$$
\end{example}

As in our previous cases (of `level 1' \cite{IKT2014} and `level 2' \cite{KPT}), we may define multi-poly-Euler numbers as follows
and connect them to values of the function $\lambda$ at negative integer arguments.

\begin{definition}\label{Def-3-1} 
For $k_1,\ldots,k_r\in \mathbb{Z}_{\geq 1}$, define multi-poly-Euler numbers $\{\ee_n^{(k_1,\ldots,k_r)}\}$ by
\begin{equation}
\frac{\aa (k_1,\ldots,k_r;\tanh\left(t/2+\pi i/4\right))}{\cosh t}=\sum_{n=0}^\infty \ee_n^{(k_1,\ldots,k_r)}\,\frac{t^n}{n!}. \label{Def-Poly-Euler}
\end{equation}
\end{definition}

\begin{remark}\label{Rem-Poly-Euler} 
The case of $r=1$ and $k_1=1$ becomes
\begin{equation*}
\frac{t}{\cosh t}=\sum_{n=0}^\infty \ee_n^{(1)}\,\frac{t^n}{n!},
\end{equation*}
namely $\ee_{m+1}^{(1)}=(m+1)E_m$ $(m\in \mathbb{Z}_{\geq 0})$, $E_m$ being Euler numbers defined in \eqref{euler}.
Also it should be noted that 
$$\frac{t}{\cosh t}=-2\sum_{a=1}^{4}\frac{\chi_4(a)te^{at}}{e^{4t}-1}=-2\sum_{n=0}^\infty B_{n,\chi_4}\frac{t^n}{n!},$$
where $B_{n,\chi_4}$ is the generalized Bernoulli number associated to the character $\chi_4$ of conductor 4. 
Hence $\ee^{(k)}_{n}$ is also regarded as the generalized poly-Bernoulli number of conductor $4$.
\end{remark}

\begin{remark}
Sasaki \cite{Sasaki2012}, about a decade ago, defined the poly-Euler numbers by using a different generating function as
\begin{equation}
\frac{{\rm Li}_{k}(1-e^{-4t})}{4t\cosh t}=\sum_{n=0}^\infty E_n^{(k)}\frac{t^n}{n!}, \label{Def-poly-Euler}
\end{equation}
and studied their properties together with Ohno in Ohno-Sasaki \cite{OS2012,OS2013,OS2017}. 
When $k=1$, this is 
$$\frac{1}{\cosh t}=\sum_{n=0}^\infty E_n^{(1)}\frac{t^n}{n!}$$
and $E_n^{(1)}$ coincides with the classical Euler number $E_n$. 
Also Hamahata \cite{Hama2014} defined the poly-Euler polynomials along the same line.

In \cite{OS2017}, Ohno and Sasaki further defined the zeta function of Arakawa-Kaneko type by
$$L_k(s)=\frac{1}{\Gamma(s)}\int_0^\infty t^{s-1}\frac{\Li_k(1-e^{-4t})}{8\cosh t}\,dt\quad ({\rm Re}(s)>0,\ k\in \mathbb{Z}_{\geq 1})$$
which satisfies
$$L_k(-n)=\frac{(-1)^n n E_{n-1}^{(k)}}{2}\quad (n\in \mathbb{Z}_{\geq 1}).$$
However, as far as the authors realize, it is unclear whether $L_k(s)$ connects poly-Euler numbers and certain multiple series as $\xi$ and $\psi$ functions did.

Poly-Bernoulli numbers $\{B_n^{(-k)}\}$ and $\{C_n^{(-k)}\}$ satisfy a kind of duality relations (\cite{Kaneko1997, Kaneko2010}), while those for poly-Euler numbers $\{E_n^{(-k)}\}$ (defined by \eqref{Def-poly-Euler}) are unknown. 
From the viewpoint of the current paper, we shall define poly-Euler numbers with non-positive indices $\{ \ee_n^{(-k)}\}$ $(k\geq 0)$ and a related zeta function, 
and study their properties in our forthcoming paper \cite{KKT2022} with Komori.
\end{remark}

\begin{remark}\label{Poly-Euler-parity} 
From \eqref{deriv-aa} and 
$$\aa (k_1,\ldots,k_{r};h(0))=\aa (k_1,\ldots,k_{r};i)=0,$$
we can show by induction that the identity
\begin{equation}
\aa (k_1,\ldots,k_r;h(-x))=(-1)^{k_1+\cdots+k_r}\aa (k_1,\ldots,k_r;h(x)) \label{Parity-aa-2}
\end{equation}
holds (take the derivatives of both sides). Hence,  if $n$ and $k_1+\cdots+k_r$ are of different parity then 
$\ee_n^{(k_1,\ldots,k_r)}=0$.
\end{remark}

\begin{example}\label{Exam-3-3}\ 
Consider the case of $r=1$ and $k_1=2$. 
As in the proof of \cite[Theorem 3.1]{Kaneko1997}, using \eqref{tau-deri}, we have
\begin{align*}
\frac{\aa(2;h(t))}{\cosh t}&=-\frac{i}{\cosh t}\int_0^{t}\frac{\aa(1;h(v))}{\cosh v}\,dv\\
& =-i\sum_{m=0}^{\infty}\ee_m^{(1)} \frac{t^{m-1}}{m!}\sum_{n=0}^\infty \ee_n^{(1)}\frac{t^{n+1}}{(n+1)!}\\
& =-i\sum_{N=0}^\infty \sum_{j=0}^{N}\binom{N}{j}\ee_{N-j}^{(1)}\frac{\ee_{j}^{(1)}}{j+1}\,\frac{t^N}{N!}.
\end{align*}
Hence 
$$\ee_N^{(2)}=-i\sum_{j=0}^{N}\binom{N}{j}\frac{\ee_{N-j}^{(1)}\ee_{j}^{(1)}}{j+1}\quad (N\in \mathbb{Z}_{\geq 0}).$$
Since $\ee_0^{(1)}=0$ and $\ee_{N}^{(1)}=NE_{N-1}$ $(N\geq 1)$, we have $\ee_{2m}^{(1)}=0$ $(m\in \mathbb{Z}_{\geq 0})$ and
$$\ee_{1}^{(1)}=1,\ \ee_{3}^{(1)}=-3,\ \ee_{5}^{(1)}=25,\ \ee_{7}^{(1)}=-427,\ \ldots$$
Therefore we have $\ee_{2m+1}^{(2)}=0$ $(m\in \mathbb{Z}_{\geq 0})$ and
$$\ee_0^{(2)}=0,\ \ \ee_2^{(2)}=-i,\ \ \ee_4^{(2)}=9i,\ \ \ee_6^{(2)}=-145i,\ \ldots$$
\end{example}

Using the well-known method of contour integration (see, for example, Washington \cite[Theorem 4.2]{Wash1997}), and noting Lemma \ref{Lem-3-5}, we can easily establish the following.

\begin{prop}\label{Prop-3-9}
For $k_1,\ldots,k_r\in \mathbb{Z}_{\geq 1}$, $\lambda(k_1,\ldots,k_r;s)$ can be analytically continued to the whole complex plane and satisfies
$$\lambda(k_1,\ldots,k_r;-n)=(-1)^n \ee_n^{(k_1,\ldots,k_r)}\quad (n\in \mathbb{Z}_{\geq 0}).$$
\end{prop}

\section{Relations between the $\lambda$-function and the one variable multiple $\TT$-function}\label{sec:zetafunct}

As in \cite{AK1999}, we consider the one variable function 
$\TT(k_1,\ldots,k_{r-1},s)$ defined by replacing $k_r$ with a variable $s$ in \eqref{ttilde-def}: 
\begin{equation*}\label{ttilde-funct} \TT(k_1,\ldots,k_{r-1},s):=
2^r\sum_{1\leq m_1<\cdots<m_r \atop m_j \equiv j \bmod 2} \frac{(-1)^{(m_r-r)/2}}{m_1^{k_1}\cdots m_{r-1}^{k_{r-1}}m_r^{s}}. \end{equation*}
Then, based on the following integral expression, we can obtain in an almost similar manner as in the previous cases
exactly the same relations among
this function, $\lambda$-function, and the $\TT$-values. Since the proofs are similar, we only give brief outlines.

\begin{lemma}[cf. \cite{AK1999}\,Theorem 3;\ \cite{KT2020-ASPM}\,Lemma 5.4] \label{Lem-4-1} For $l_1,\ldots,l_{r-1}\in \mathbb{Z}_{\geq 1}$ and ${\rm Re}(s)>1$,  
\begin{align}
\TT(l_1,\ldots,l_{r-1},s) &= \frac{1}{\Gamma(l_1)\cdots\Gamma(l_{r-1})\Gamma(s)}\int_{0}^\infty \cdots \int_{0}^\infty x_1^{l_1-1}\cdots x_{r-1}^{l_{r-1}-1}x_r^{s-1}\notag\\
& \qquad \times \prod_{j=1}^{r}\frac{1}{\cosh(x_j+\cdots+x_r)}dx_1\cdots dx_r. \label{TT-integral}
\end{align}
\end{lemma}

\begin{proof}
Exactly the same method as in the proof of \cite[Theorem 3]{AK1999} using
\begin{align*}
& \frac{1}{\cosh(x_j+\cdots+x_r)}=\frac{2e^{-x_j-\cdots-x_r}}{1+e^{-2(x_j+\cdots+x_r)}} =2\sum_{m_j=0}^\infty (-1)^{m_j}e^{-(2m_j+1)(x_j+\cdots+x_r)}
\end{align*}
works, and the iterated integral can be computed to obtain the assertion.
\end{proof}

\begin{theorem}[cf. \cite{AK1999}\,Theorem 8;\ \cite{KT2020-ASPM}\,Theorem 5.3] \label{Th-4-2} For $r, k\in \mathbb{Z}_{\geq 1}$, 
\begin{align}
& \lambda(\underbrace{1,\ldots,1}_{r-1},k;s) \notag  \\
&=(-1)^{k-1}i^{1-k}\sum_{a_1,\ldots,a_k\geq 0 \atop a_1+\cdots+a_k=r}
\binom{s+a_k-1}{a_k}\cdot \TT(a_1+1,\ldots,a_{k-1}+1,a_k+s)\notag\\
&\quad +i^{1-k}\sum_{j=0}^{k-2}(-1)^{j}\,\TT(\underbrace{1,\ldots,1}_{k-2-j},r+1)\cdot 
\TT(\underbrace{1,\ldots,1}_{j},s).\label{eq-Th-4-2}
\end{align}
\end{theorem}

\begin{proof}
The method of the proof is similar to that of \cite[Theorem 8]{AK1999} and \cite[Theorem 5.7]{KT2020-ASPM} (see also \cite[Theorem 7]{Sasaki2012}). 
Given $r,k\ge1$, introduce the following integral 
\begin{align*}
J_\nu^{(r,k)}(s)& =\frac{1}{\Gamma(s)}\! \int_{0}^\infty \!\!\!\cdots\! \int_{0}^\infty 
\frac{\aa(\overbrace{1,\ldots,1}^{r-1},\nu;h(x_\nu+\cdots+x_k))}{\prod_{l=\nu}^{k}
\cosh(x_l+\cdots+x_k)} x_k^{s-1}\,dx_\nu\cdots dx_k\quad (1\leq \nu \leq k). 
\end{align*}
We compute $J_1^{(r,k)}(s)$ in two different ways.  First, since 
\[ \aa(\underbrace{1,\ldots,1}_{r};h(x_1+\cdots+x_k))=\frac{\aa(1;h(x_1+\cdots+x_k))^r}{r!}=
\frac{(x_1+\cdots+x_k)^r}{r!}  \]
by the shuffle product and Lemma~\ref{defofh}, we have
\begin{align*}
J_1^{(r,k)}(s)& =\frac{1}{\Gamma(s)\, r!}\int_{0}^\infty \cdots \int_{0}^\infty 
\frac{(x_1+\cdots+x_k)^rx_k^{s-1}}{\prod_{l=1}^{k}\cosh(x_l+\cdots+x_k)}\,dx_1\cdots dx_k\\
& =\frac{1}{\Gamma(s)}\!\!\sum_{a_1+\cdots+a_k=r}
\frac{1}{a_1!\cdots a_k!}\int_{0}^\infty\!\!\! \cdots\! \int_{0}^\infty x_1^{a_1}\cdots x_{k-1}^{a_{k-1}}x_k^{s+a_k-1}\\
& \qquad \times \frac1{\prod_{l=1}^{k}\cosh(x_l+\cdots+x_k)}\,dx_1\cdots dx_k\\
&=\sum_{a_1+\cdots+a_k=r}\frac{\Gamma(s+a_k)}{\Gamma(s)a_k!}
\times \frac1{\Gamma(a_1+1)\cdots\Gamma(a_{k-1}+1)\Gamma(s+a_k)}\\
&\qquad\times \int_{0}^\infty\!\!\! \cdots\! \int_{0}^\infty 
\frac{x_1^{a_1}\cdots x_{k-1}^{a_{k-1}}x_k^{s+a_k-1}}
{\prod_{l=1}^{k}\cosh(x_l+\cdots+x_k)}\,dx_1\cdots dx_k.
\end{align*}
Using Lemma~\ref{Lem-4-1} for the last integral, we obtain
\begin{align}
 \label{I1-1}   J_1^{(r,k)}(s) &=   \sum_{a_1+\cdots+a_k=r}\binom{s+a_k-1}{a_k} \cdot \TT(a_1+1,\ldots,a_{k-1}+1,a_k+s).
\end{align}

Secondly, by Lemma \ref{defofh}, 
we compute
\begin{align*}
J_\nu^{(r,k)}(s)
&=\frac{2^r}{\Gamma(s)}  \int_{0}^\infty\!\!\! \cdots\! \int_{0}^\infty 
\left[ i\aa(\underbrace{1,\ldots,1}_{r-1},\nu+1;h(x_\nu+\cdots +x_k))
\right]_{x_\nu=0}^\infty\\
&\qquad \times \frac1{\prod_{l=\nu+1}^k \cosh(x_l+\cdots+x_k)}\, x_k^{s-1}\,dx_{\nu+1}\cdots dx_k\\
&=i\aa(\underbrace{1,\ldots,1}_{r-1},\nu+1;1)\cdot
\TT(\underbrace{1,\ldots,1}_{k-\nu-1},s) -i\,J_{\nu+1}^{(r,k)}\\
&=i\aa(\underbrace{1,\ldots,1}_{r-1},\nu+1;1)\cdot
\TT(\underbrace{1,\ldots,1}_{k-\nu-1},s)\\
&\ \ -i^2\aa(\underbrace{1,\ldots,1}_{r-1},\nu+2;1)\cdot
\TT(\underbrace{1,\ldots,1}_{k-\nu-2},s) +i^2\,J_{\nu+2}^{(r,k)}.
\end{align*}
Therefore, repeating this operation, we obtain 
\begin{align}
J_1^{(r,k)}(s)&=\sum_{\nu=1}^{k-1}(-1)^{\nu-1}i^\nu \aa(\underbrace{1,\ldots,1}_{r-1},\nu+1;1)\cdot \TT(\underbrace{1,\ldots,1}_{k-\nu-1},s)+(-1)^{k-1}i^{k-1}J_{k}^{(r,k)}(s) \notag \\
&=\sum_{j=0}^{k-2}(-1)^{k-j}i^{k-j-1}\aa(\underbrace{1,\ldots,1}_{r-1},k-j;1)\TT(\underbrace{1,\ldots,1}_{j},s)\notag \\
&\qquad +(-1)^{k-1}i^{k-1}\lambda(\underbrace{1,\ldots,1}_{r-1},k;s),\label{I1-2}
\end{align}
by setting $j=k-\nu-1$ and noting
\[ J_k^{(r,k)}(s)=\lambda(\underbrace{1,\ldots,1}_{r-1},k;s). \]
Comparing \eqref{I1-1} and \eqref{I1-2}, we obtain the assertion.
\end{proof}

\begin{theorem}[cf. \cite{AK1999}\ Theorem 9\ (i);\ \cite{KT2020-ASPM}\ Theorem 5.5]\label{T-5-1}\ For $r,k\in \mathbb{Z}_{\geq 1}$ and $m\in \mathbb{Z}_{\geq 0}$,
\begin{align}
 &\lambda(\underbrace{1,\ldots,1}_{r-1},k;m+1)\notag\\
 &=i^{1-k}\sum_{a_1,\ldots,a_k\geq 0 \atop a_1+\cdots+a_k=m}\binom{a_k+r}{r}
 \cdot \TT(a_1+1,\ldots,a_{k-1}+1,a_k+r+1). \label{ee-5-4}
\end{align}
\end{theorem}

\begin{proof}
By Lemma \ref{defofh}, we have
\begin{align*}
& \lambda(\underbrace{1,\ldots,1}_{r-1},k;m+1)\\
&=\frac{1}{m!i}\int_0^\infty \frac{t_k^{m}}{\cosh t_k} \int_0^{t_k}
\frac{\aa(\overbrace{1,\ldots,1}^{r-1},k-1;h(t_{k-1})}{\cosh t_{k-1}}
\, dt_{k-1}dt_k\\
&=\frac{1}{m!i^2}\int_0^\infty \frac{t_k^{m}}{\cosh t_k} \int_0^{t_k}
\frac{1}{\cosh t_{k-1}}\int_0^{t_{k-1}}\\
&\qquad\qquad\frac{\aa(\overbrace{1,\ldots,1}^{r-1},k-2;h(t_{k-2}/2))}{\cosh t_{k-2}}
\, dt_{k-2}dt_{k-1}dt_k\\
&=\cdots\\
&=\frac{1}{m!i^{k-1}}\int_0^\infty \int_0^{t_k}\cdots \int_0^{t_2}
\frac{t_k^m\, \aa(\overbrace{1,\ldots,1}^{r};h(t_1))}
{\cosh(t_k)\cdots\cosh(t_1)}\, dt_1\cdots dt_k\\
&=\frac{i^{1-k}}{m!r!}\int_0^\infty \int_0^{t_k}\cdots \int_0^{t_2}
\frac{t_k^m\,t_1^r}
{\cosh(t_k)\cdots\cosh(t_1)}\, dt_1\cdots dt_k.
\end{align*}
By the change of variables 
\[t_1=x_k, t_2=x_{k-1}+x_k,\ldots, t_k=x_1+\cdots+x_k,\]
and using Lemma \ref{Lem-4-1}, 
we obtain 
\begin{align*}
&\lambda(\underbrace{1,\ldots,1}_{r-1},k;m+1)\\
&=\frac{i^{1-k}}{m!r!}\int_0^\infty\cdots \int_0^\infty
\frac{(x_1+\cdots+x_k)^m\,x_k^r}
{\prod_{l=1}^k\cosh(x_l+\cdots+x_k)}\, dx_1\cdots dx_k\\
&=i^{1-k}\sum_{a_1+\cdots+a_k=m}\binom{a_k+r}{r}
 \cdot \TT(a_1+1,\ldots,a_{k-1}+1,a_k+r+1). 
\end{align*}
\end{proof}

Setting $s=m+1$ in Theorem~\ref{Th-4-2} and comparing with Theorem~\ref{T-5-1}, 
we obtain the following\footnote{Ce Xu has kindly
informed us of their recent paper \cite{XZ} with Jianqiang Zhao, where they study more general
level 4 multiple $L$-values and obtain a similar theorem (\cite[Theorem~4.4]{XZ}) in their
setting, of which the following is a special case.} which corresponds to \cite[Corollary 11]{AK1999}.

\begin{theorem}[cf. \cite{AK1999}\ Corollary 11;\ \cite{KT2020-ASPM}\ Theorem 5.7] \label{Th-4-7}\ For $m,r\geq 1$ and $k\ge2$,
\begin{align}
&\sum_{a_1,\ldots,a_k\geq 0 \atop a_1+\cdots+a_k=m}\binom{a_k+r}{r}\cdot \TT(a_1+1,\ldots,a_{k-1}+1,a_k+r+1) \notag\\
& \ +(-1)^k \sum_{a_1,\ldots,a_k\geq 0 \atop a_1+\cdots+a_k=r}\binom{a_k+m}{m}\cdot \TT(a_1+1,\ldots,a_{k-1}+1,a_k+m+1)\notag\\
& =\sum_{j=0}^{k-2}(-1)^j \TT(\underbrace{1,\ldots,1}_{k-j-2},r+1)\cdot \TT(\underbrace{1,\ldots,1}_{j},m+1). \label{eq-Th-4-7}
\end{align}
\end{theorem}

{\bf Acknowledgements.}\ 
{The authors are very grateful to Minoru Hirose, who provided a Pari-GP program to numerically  compute $\TT$-values efficiently. 
They also express their gratitude to him and Ryota Umezawa for pointing out the proof of Theorem~\ref{Teven}.
This work was supported by JSPS KAKENHI Grant Numbers JP16H06336, JP21H04430 (Kaneko), and JP21K03168  (Tsumura).}

\ 

M.~Kaneko \endgraf
Faculty of Mathematics, Kyushu University \endgraf
Motooka 744, Nishi-ku, Fukuoka 819-0395 Japan\endgraf
{\rm email}:\,kaneko.masanobu.661@m.kyushu-u.ac.jp

\ 

H.~Tsumura \endgraf
Department of Mathematical Sciences, Tokyo Metropolitan University \endgraf
1-1, Minami-Ohsawa, Hachioji, Tokyo 192-0397 Japan\endgraf
{\rm email}:\,tsumura@tmu.ac.jp



\begin{thebibliography}{3}

\bibitem{Ao} K.~Aomoto, Special values of hyperlogarithms and
  linear difference schemes, Illinois J. of Math. {\bf 34-2} (1990), 191--216.

\bibitem{AK1999}  
T.~Arakawa and M.~Kaneko, Multiple zeta values, poly-Bernoulli numbers, and related zeta functions,  Nagoya Math. J., {\bf 153} (1999), 189--209.

\bibitem{AK2004}
T.~Arakawa and M.~Kaneko, On multiple $L$-values, J. Math. Soc. Japan {\bf 56} (2004), 967--991.

\bibitem{Chen}
K.-T.~Chen,  Iterated integrals of differential forms and loop space homology, Ann. of Math. (2) {\bf 97} (1973), 217--246.

\bibitem{Comtet1974}
L.~Comtet, Advanced combinatorics, The art of finite and infinite expansions, Revised and enlarged edition, D. Reidel Publishing Co., Dordrecht, 1974. 

\bibitem{Dr} 
V.~G.~Drinfel'd, {On quasitriangular quasi-Hopf algebras and a group closely connected with Gal$(\bar \Q/\Q)$}, Leningrad Math. J. {\bf 2} (1991), 829--860.

\bibitem{Entringer1666}
R.~C.~Entringer, A combinatorial interpretation of the Euler and Bernoulli numbers, Nieuw Arch. Wisk. (3) {\bf 14} (1966), 241--246. 

\bibitem{FY}
S.~Fukuhara and Y.~Yang, Period polynomials and explicit formulas for Hecke operators on $\Gamma_0(2)$. Math. Proc. Cambridge Philos. Soc. {\bf 146} (2009), no. 2, 321--350.

\bibitem{FY2}
S.~Fukuhara and Y.~Yang, A basis for $S_k(\Gamma_0(4))$ and representations of integers as sums of squares, Ramanujan J. {\bf 28} (2012), no. 1, 25--43.

\bibitem{GKZ} 
H.~Gangl, M.~Kaneko, and D.~Zagier, Double zeta values and modular forms, `Automorphic forms and Zeta functions', 
Proceedings of the conference in memory of Tsuneo Arakawa, World Scientific, (2006), 71--106. 

\bibitem{Hama2014}
Y.~Hamahata, Poly-Euler polynomials and Arakawa-Kaneko type zeta functions, Functiones et Approximatio, 5.1.1 (2014),  7--22.

\bibitem{IKT2014}  
K.~Imatomi, M.~Kaneko and E.~Takeda, Multi-poly-Bernoulli numbers and finite multiple zeta values, J. Integer Sequences, {\bf 17} (2014), Article 14.4.5.

\bibitem{Kaneko1997}  
M.~Kaneko, Poly-Bernoulli numbers, J. Th\'eor. Nombres Bordeaux, {\bf 9} (1997), 199--206.

\bibitem{Kaneko2010} 
M.~Kaneko, Poly-Bernoulli numbers and related zeta functions, in `Algebraic and Analytic Aspects of Zeta Functions and $L$-functions', 
G. Bhowmik et al. (eds.), MSJ Memoirs {\bf 21}, Math. Soc. Japan, (2010), 73--85.

\bibitem{KKT2022}
M.~Kaneko, Y.~Komori and H.~Tsumura, On Arakawa-Kaneko zeta-functions associated with ${\rm GL}_2(\mathbb{C})$ and their functional relations II, in preparation.

\bibitem{KPT}  
M.~Kaneko, M.~Pallewatta, and H.~Tsumura, On poly-cosecant numbers, J. Integer Sequences, {\bf 23} (2020), Article 20.6.4.

\bibitem{KT2018}
M.~Kaneko and H.~Tsumura, Multi-poly-Bernoulli numbers and related zeta functions, Nagoya Math. J. {\bf 232} (2018), 19--54.

\bibitem{KT2020-ASPM}
M.~Kaneko and H.~Tsumura, 
Zeta functions connecting multiple zeta values and poly-Bernoulli numbers, Adv. Stud. Pure Math. {\bf 84}, 2020, pp. 181--204.

\bibitem{KT2020-Tsukuba}
M.~Kaneko and H.~Tsumura, On multiple zeta values of level two, Tsukuba J. Math. {\bf 44} (2020), 213--234.


\bibitem{Miyake}
T.~Miyake, Modular forms, Translated from the 1976 Japanese original by Yoshitaka Maeda, Reprint of the first 1989 English edition, 
Springer Monographs in Mathematics, Springer-Verlag, Berlin, 2006.

\bibitem{N1924} 
N.~E.~N\"orlund,  Vorlesungen \"uber Differenzenrechnung, Springer-Verlag, Berlin, 1924.

\bibitem{OS2012}
Y.~Ohno and Y.~Sasaki, On the party of poly-Euler numbers, RIMS K$\hat{\rm o}$k$\hat{\rm u}$roku Bessatsu, B32, (2012), 271--278.

\bibitem{OS2013}
Y.~Ohno and Y.~Sasaki, Periodicity on poly-Euler numbers and Vandiver type congruence for Euler numbers, RIMS K$\hat{\rm o}$k$\hat{\rm u}$roku Bessatsu, B44, (2013), 205--211.

\bibitem{OS2017}
Y.~Ohno and Y.~Sasaki, On poly-Euler numbers, J. Aust. Math. Soc. {\bf 103} (2017), 126--144. 

\bibitem{Sasaki2012}
Y.~Sasaki, On generalized poly-Bernoulli numbers and related $L$-functions, J. Number Theory, {\bf 132} (2012), 156--170.

\bibitem{Stanley2010}
R.~P.~Stanley, A survey of alternating permutations, Combinatorics and graphs, 165--196, Contemp. Math., 531, Amer. Math. Soc., Providence, RI, (2010).

\bibitem{Ts2005} 
H.~Tsumura, Certain functional relations for the double harmonic series related to the double Euler numbers, J. Aust. Math. Soc. {\bf 79} (2005), 319--333. 

\bibitem{Ts2007} 
H.~Tsumura, On the parity conjecture for multiple $L$-values of conductor four, Tokyo J. Math. {\bf 30} (2007), 21--40. 

\bibitem{U}
R.~Umezawa, Multiple $T$-values and iterated log-tangent integrals, in preparation.

\bibitem{Wash1997} 
L.~C.~Washington, Introduction to cyclotomic fields, 2nd ed., Springer-Verlag: New York Berlin Heidelberg, 1997.

\bibitem{WW}
E.~T.~Whittaker and G.~N.~Watson, A course of modern analysis, reprint of the fourth (1927) edition, 
Cambridge University Press, Cambridge, 1996.

\bibitem{XZ}
C.~Xu and J.~Zhao, Alternating multiple $T$-values: weighted sums, duality, and dimension conjecture, to appear in Ramanujan J.

\end{thebibliography}
\end{document}